\documentclass{article}
\usepackage{amsmath,amssymb,amsthm}

\title{Homology and orientation reversing periodic maps on surfaces}

\author{Haibin Hang}

\newtheorem{theorem}{Theorem}[section]
\newtheorem{corollary}[theorem]{Corollary}
\newtheorem{lemma}[theorem]{Lemma}
\newtheorem{definition}{Definition}[section]

\newcommand{\Fix}{\operatorname{Fix}}
\newcommand{\lcm}{\operatorname{lcm}}

\begin{document}

\begin{abstract}
In this paper, we give a classification of orientation reversing periodic maps on closed surfaces which generalizes the theory of Nielsen for the orientation preserving periodic maps.

On one hand, we give a collection of data for each orientation reversing periodic map such that two periodic maps with the same data must be conjugate to each other. On the other hand, we give the criterion to judge when two different collections of data correspond to the same conjugacy class.

As an application of the results of this paper, we shall show that a given orientation reversing periodic map on $\Sigma_g$ with period larger than or equal to $3g$ must be conjugate to the power of a list of particular types of periodic maps.
\end{abstract}
\maketitle


\section{Introduction}\label{introduction}
A natural idea to consider about the conjugacy class of periodic map $f$ with period $n$ on compact surface $\Sigma$ is to consider about its induced branched covering $\pi_f$ from $\Sigma$ to the orbit space $\Sigma/f$. Let $B_f$ be the set of branch points, there is a well-defined homomorphism $\omega_f\colon\thinspace H_1(\Sigma/f-B_f)\rightarrow Z_n$. In a paper of Nielsen~\cite{Nielsen}, he proved that if $f$ is orientation preserving, the conjugacy class of $f$ is determined by the values of $\omega_f$ at the boundary loops of $\Sigma/f-N(B_f)$, where $N(B_f)$ is a tubular neighborhood of $B_f$.

Afterwards, the classification of symmetries on surfaces is encompassed in the study of the structure of NEC groups, which is based on the classification theorem of A.M. Macbeath~\cite{Macbeath0} and his approach to surface automorphisms as explained in his paper~\cite{Macbeath1}. This approach was used by E. Bujalance and D. Singerman~\cite{Bujalance0} to classify period 2 orientation reversing automorphisms on Riemann surfaces. Then in E. Bujalance and A.F. Costa's paper~\cite{Bujalance1}, they characterized the conjugacy classes of orientation reversing periodic automorphisms on Riemann surfaces with period $2p$, here $p$ a prime integer.

Some other facts we know is given by J.J. Etayo~\cite{Etayo} and S.C. Wang~\cite{Wang}: If $g\geq 2$ is even, the largest possible period of an orientation reversing periodic map on $\Sigma_g$ is attainable by $4g+4$; If $g\geq 2$ is odd, the largest possible period of an orientation reversing periodic map on $\Sigma_g$ is attainable by $4g-4$.

In this paper we focus on how to expend Nielsen's idea to orientation reversing periodic maps, in which case $\Sigma/f$ may be nonoreintable. We find out that the conjugacy class of $f$ is determined by the values of $\omega_f$ at a set of one-sided loops on $\Sigma/f-N(B_f)$ besides with the boundary loops of $\Sigma/f-N(B_f)$. From this point of view, we build up a complete theory about orientation reversing periodic maps. Particularly, the corollary \ref{local} would greatly generalize the results in the paper~\cite{Bujalance0} and~\cite{Bujalance1} mentioned above. And in section \ref{large period}, we would give a detailed depict of orientation reversing periodic maps with period $n>2(g-1)$ and give a list of all those with period $n\geq 3g$ up to conjugacy. An interesting result I would like to mention here is corollary \ref{4g-4}: if $g>3$ is odd, all the possible orientation reversing periodic maps with period $n\geq 3g$ are conjugate to a particular map with period $4g-4$.

The author would like to express his gratitude to Professor Xuezhi Zhao, for his guidance during the author's master's degree. Also thanks Professor G. Gromadzki for introducing the results of E. Bujalance and A.F. Costa~\cite{Bujalance1}.
\section{Preliminaries}\label{preliminary}
Along this paper, we use symbols like $\Sigma$, $\Sigma_g$ and $\Sigma_{g,b}$ to represent an oriented compact surface, an oriented closed surface with genus $g$ and an oriented compact surface with genus $g$ and boundary components number $b$. We use symbols like $N_{\tau,b}$ to represent a non-orientable compact surface with genus $\tau$ and boundary components number $b$. And we use $N(\cdot)$ to denote a tubular neighborhood of some subset on the underground surface.

\begin{definition}
A homeomorphism $f$ from a compact surface $\Sigma$ to itself is said to be a periodic map of period $n$ if $f^n=id_{\Sigma}$ and $n$ is the smallest positive integer that satisfies this condition.
\end{definition}

Along this paper, we always assume $f\neq id_{\Sigma}$, which means $n>1$.

\begin{definition}
A point $x$ on $\Sigma$ is called a multiple point of $f$ if there is an integer $0<t<n$ such that $f^t(x)=x$. If $t$ is the smallest possible integer then $t|n$.  We call $t$ the orbit length of $x$ and we call $n/t$ the multiplicity of $x$.
\end{definition}

In the following, we denote the set of points with orbit length $t$ by $W_f^t$. And let $W_f$ be the set of all multiple points of $f$, then $W_f=\sqcup_{t<n}W_f^t$.

It's well-known that $f$ can be realized as an isometry under some constant curvature metric on the surface $\Sigma$. Then locally at each fixed point, $f$ is a rotation or reflection. And we have the following two lemmas:

\begin{lemma}\label{fixed points1}
If $f$ preserves orientation and $\Fix(f)\neq\emptyset$, then each point of $\Fix(f)$ is isolated and $\#\Fix(f)$ is finite.
\end{lemma}

\begin{lemma}\label{fixed points2}
If $f$ reverses orientation and $\Fix(f)\neq\emptyset$, then $\Fix(f)$ is disjoint union of simple closed curves and the period of $f$ must be $2$.
\end{lemma}

Given a multiple point $x\in W_f^t$, it's clear that $x$ is a fixed point of $n/t$-periodic map $f^t$. Thus we have:

\begin{corollary}\label{multiple points1}
If $f$ preserves orientation and $W_f\neq\emptyset$, then $W_f$ contains only finite number of points.
\end{corollary}

If $f$ reverses orientation, then $f^t$ preserves (reverses) the orientation when $t$ is even (odd). Thus we have:
\begin{corollary}\label{multiple points2}
Suppose that $f$ reverses orientation. If $t$ is even and $W_f^t\neq\emptyset$, then it
contains only finite number of points; If $t$ is odd and $W_f^t\neq\emptyset$, then it is disjoint union of simple closed curves and the period of $f$ must be $2t$.
\end{corollary}

\begin{definition}
For two periodic maps $f$ on $\Sigma$ and $f'$ on $\Sigma'$, if there exists an orientation preserving
homeomorphism $h\colon\thinspace \Sigma\rightarrow\Sigma'$ such that $h(W_f)=W_{f'}$ and $f'\circ h=h\circ f$ holds on a nonempty tubular neighborhood $N(W_f\cup\partial\Sigma)$, we say that $f$ and $f'$ are locally conjugate. If furthermore, $f'\circ h=h\circ f$ holds on $\Sigma$, we say that $f$ and $f'$ are conjugate.
\end{definition}

The quotient space $\Sigma/f$ obtained by identifying $x$ with $f(x)$ is also a compact surface. Let $\pi_f\colon\thinspace\Sigma\rightarrow\Sigma/f$ be the quotient map. Then $\pi_f$ is an $n$-fold branched covering ramified at $B_f=\pi_f(W_f)$.

\begin{definition}
For any simple loop $\gamma$ on $\Sigma/f-B_f$, $\pi_f^{-1}(\gamma)$ is disjoint union of simple loops and we call the components number $t$ the orbit length of $\gamma$. Then each component $\tilde{\gamma}\subset\Sigma$ is an $n/t$ covering of $\gamma$ and we call $q=n/t$ the multiplicity of the loop $\gamma$ or $\tilde{\gamma}$. The map $f^t$ restricted on $\tilde{\gamma}$ is conjugate to $e^{i\theta}\mapsto e^{i(\theta+2\pi\frac{p}{q})}$ for some $p$, here $up-vq=1$ for some $u=p^{-1} \mod q$. We call $p/q$ the rotation number of $f^t$. And $[t,q,u]$ is called the the valency of $\gamma$ or $\tilde{\gamma}$.
\end{definition}

We choose a point $x$ in $\Sigma/f-B_f$ and a point $\tilde{x}$ in $\pi_f^{-1}(x)$, then define a homomorphism $\Omega_f\colon\thinspace\pi_1(\Sigma/f-B_f,x)\rightarrow Z_n$ as follows:

For any loop $\gamma\subset\Sigma/f-B_f$ with the base point $x$, let $[\gamma]$ be the element of $\pi_1(\Sigma/f-B_f,x)$ represented by $\gamma$. Let $\tilde{\gamma}$ be the lift of $\gamma$ on $\Sigma$ which begins from $\tilde{x}$. There is a positive integer $r$ less than or equal to $n$ such that the terminal point of $\tilde{\gamma}$ is $f^r(\tilde{x})$. We define $\Omega_f([\gamma])=r\mod n$. Since $Z_n$ is an abelian group, the homomorphism $\Omega_f$ induces a homomorphism $\omega_f\colon\thinspace H_1(\Sigma/f-B_f)\rightarrow Z_n$.

\begin{lemma}\label{index}
For any simple loop $\gamma\subset\Sigma/f-B_f$ with valency $[t,q,u]$, we have $t=\gcd\{\omega_f(\gamma),n\}$, $q=n/t$, $u=\frac{\omega_f(\gamma)}{t}$.
\end{lemma}

\begin{proof}
	Since $f^t$ restricted on $\tilde{\gamma}$ has rotation number $p/q$ and $up-vq=1$, we have $utp-vn=t$ and $t=\gcd\{ut,n\}$. Then $f^{tu}$ restricted on $\tilde{\gamma}$ has rotation number $1/q$ and maps $\tilde{\gamma}(0)$ to $\tilde{\gamma}(1)$. Hence $\omega_f(\gamma)=ut$ and the statement follows.
\end{proof}

In the following sections we would show how a collection of data derived from homomorphism $\omega_f$ can determine the conjugacy class of $f$.

\section{The conjugacy class of orientation preserving periodic maps}\label{orientable}
In this section we look at the case when $f$ preserves orientation of a bounded oriented surface $\Sigma$. In this case, $\Sigma/f$ is also a bounded surface with induced orientation from $\Sigma$.

By corollary \ref{multiple points1}, $B_f=\pi_f(W_f)$ is finit. We let  $B_f=\{q_1,q_2,\cdots,q_b\}$ and let $\partial(\Sigma/f)=\sqcup_{j=1}^c\delta_j$. Then $\Sigma/f-N(B_f)$ has $b+c$ boundary loops, each of which has induced orientation. We use $S_{q_i}$ to denote each boundary loop around $q_i$.
\begin{definition}
	For each branch point $q_i$, we refer the valency of $q_i$ to be the valency of the simple loop $S_{q_i}$ around $q_i$.
\end{definition}

\begin{lemma}\label{valency}
Given two orientation preserving periodic maps $f$ on $\Sigma$ and $f'$ on $\Sigma'$. If $B_f\cup\partial(\Sigma/f)\neq\emptyset$ and there is an orientation preserving homeomorphism $\bar h\colon\thinspace\Sigma/f\rightarrow\Sigma'/f'$, such that
 \begin{enumerate}
  \item $\bar h(B_f)=B_{f'}$
  \item for each $q_i\in B_f$, $q_i$ and $\bar h(q_i)$ have the same valency
  \item for each $\delta_j\subset\partial(\Sigma/f)$, $\delta_j$ and $\bar h(\delta_j)$ have the same valency
 \end{enumerate}
Then $f$ and $f'$ are locally conjugate.
\end{lemma}

\begin{theorem}\label{Nielsen11}\cite{Nielsen}
If two homeomorphic compact oriented surfaces undergo periodic, orientation preserving transformations $f$ and $f'$, respectively, of the same period, then $f$ and $f'$ are conjugate if and only if they are locally conjugate.
\end{theorem}

Let $h=genus(\Sigma/f)$, $\theta_i=\omega_f(S_{q_i})$ and $\eta_j=\omega_f(\delta_j)$. By lemma \ref{index}, lemma \ref{valency} and theorem \ref{Nielsen11}, the data $[h,n;(\theta_1,\cdots,\theta_b);(\eta_1,\cdots,\eta_c)]$ determines a periodic map up to conjugacy. The Hurwitz realization problem for cyclic branched covering gives us the following sufficient and necessary condition for the data $[h,n;(\theta_1,\cdots,\theta_b);(\eta_1,\cdots,\eta_c)]$ to correspond to an orientation preserving periodic map.

\begin{theorem}\label{Nielsen12}
There is an orientation preserving periodic map of period $n$ on compact orientable surface of genus $\tau$ with the data $[h,n;(\theta_1,\cdots,\theta_b);(\eta_1,\cdots,\eta_c)]$ if and only if the following conditions are satisfied:
 \begin{enumerate}
 \item $\theta_i\neq 0\mod n$, $i=1,\cdots,b$
 \item $\theta_1+\cdots+\theta_b+\eta_1+\cdots+\eta_c=0\mod n$
 \item If $h=0$, then $\gcd\{\theta_1,\cdots,\theta_b,\eta_1,\cdots,\eta_c\}=1 \mod n$
 \item $2\tau-2=n\left(2h-2+\sum_{i=1}^b\left(1-\frac{\gcd\{\theta_i,n\}}{n}\right)+\sum_{j=1}^c\left(1-\frac{\gcd\{\eta_j,n\}}{n}\right)\right)$
 \end{enumerate}
\end{theorem}

In particular, for orientation preserving periodic map $f$ on a closed surface $\Sigma_g$, we have $c=0$. Then
\begin{definition}
For each $q_i\in B_f$, let $[t_i,q_i,u_i]$ be its valency, then the expression $V=(n,u_1/q_1+\cdots+u_b/q_b)$ is called the total valency of $f$. As inferred from the above theorem, $u_1/q_1+\cdots+u_b/q_b$ is an integer.
\end{definition}

\begin{corollary}\label{m}
Suppose that $f$ has total valency $V=(n,u_1/q_1+\cdots+u_b/q_b)$ and $\Sigma_g/f$ is a sphere, then for each $1\leq i\leq b$, $\lcm\{q_1,\cdots,q_{i-1},\hat{q_i},q_{i+1},\cdots,q_b\}=n$, here $\hat{q_i}$ means $q_i$ is omitted.
\end{corollary}

\begin{proof}
By theorem \ref{Nielsen12}, $\gcd\{\theta_1,\cdots,\theta_b\}=1\mod n$. By lemma \ref{index}, we have $\theta_i=u_i\frac{n}{q_i}$ and $\gcd\{u_1\frac{n}{q_1},\cdots,u_b\frac{n}{q_b}\}=1\mod n$, which implies $\gcd\{\frac{n}{q_1},\cdots,\frac{n}{q_b}\}=1$. Thus $\lcm\{q_1,\cdots,q_b\}=n$.

By theorem \ref{Nielsen12}, $\theta_1+\cdots+\theta_b=0\mod n$, then $\Sigma_{i=1}^b\frac{u_i}{q_i}=\Sigma_{i=1}^b\frac{\theta_i}{n}$ is an integer. Then for each $i$, $\frac{u_i}{q_i}=\frac{r}{\lcm\{q_1,\cdots,q_{i-1},\hat{q_i},q_{i+1},\cdots,q_b\}}$, for some integer $r$. From $\gcd\{u_i,q_i\}=1$, we have $q_i|\lcm\{q_1,\cdots,q_{i-1},\hat{q_i},q_{i+1},\cdots,q_b\}$. Then the conclusion follows.
\end{proof}

\section{The data for orientation reversing periodic maps}\label{new data}
In the following sections we shall consider orientation reversing periodic map $f$ of period $n$ on closed surface $\Sigma_g$. Since $f$ is orientation reversing, the period $n$ must be even. We suppose that $n=2m$.

Similar to the method for the orientation preserving case, we intend to find a collection of data derived from the map $\omega_f\colon\thinspace H_1(\Sigma_g/f-B_f)\rightarrow Z_n$ to depict the conjugacy class of $f$.

Since $f^2$ preserves the orientation, the quotient map $\pi_{f^2}\colon\thinspace\Sigma_g\rightarrow\Sigma_g/f^2$ which identifies $x$ with $f^2(x)$ on $\Sigma_g$ gives us a closed oriented quotient surface. And $f$ induces a 2-periodic orientation reversing map
$\bar f\colon\thinspace\Sigma_g/f^2 \rightarrow \Sigma_g/f^2$ such that $\bar f\circ\pi_{f^2}=\pi_{f^2}\circ f$. Then we have a quotient map $\pi_{\bar f}\colon\thinspace\Sigma_g/f^2\rightarrow\Sigma_g/f$. In particular, $\pi_{\bar f}\circ\pi_{f^2}=\pi_f$.

Following lemma \ref{fixed points2}, we suppose that $\Fix(\bar f)$ (if not empty) is disjoint union of $k$ simple closed curves $\sigma_1,\cdots,\sigma_k$. Let $\delta_j=\pi_{\bar f}(\sigma_j)$, we have $\partial(\Sigma_g/f)=\sqcup_{j=1}^k \delta_j$.

Following corollary \ref{multiple points2}, we suppose the isolated branch points of $\pi_{f}$ to be $\{q_1,\cdots,q_b\}$, which is exactly $\pi_{\bar f}(B_{f^2})$.

Then $B_f=\pi_{\bar f}(\Fix(\bar f))\sqcup \pi_{\bar f}(B_{f^2})=\sqcup_{j=1}^k\delta_j\sqcup\{q_1,\cdots,q_b\}$. We denote the boundary components of $\Sigma_g/f-N(B_f)$ by $S_{q_i}$, $i=1,\cdots,b$ and $S_{\delta_j}$, $j=1,\cdots,k$.

As what we did for orientation preserving periodic maps, we can look at the value of $\omega_f$ at each $S_{q_i}$ and $S_{\delta_j}$. But the problem is that $\Sigma_g/f$ may be non-orientable, which means we do not have well-defined orientations for $S_{q_i}$ and $S_{\delta_j}$. Correspondingly, we can not decide the signs of $\omega_f(S_{q_i})$ and $\omega_f(S_{\delta_j})$ in $Z_n$.

%But we can cut $\Sigma_g/f$ along a sequence of disjoint one-sided simple closed curves $\iota_l$, $l=1,\cdots,s$ (the tubular neighborhood of each $\iota_l$ is homeomorphic to M\"{o}bius band), such that the result surface, denoted by $\Sigma_{h,k+s}$ is orientable. We can give $\Sigma_{h,k+s}$ an arbitrary orientation, which induces orientation on each of the boundary components of $\Sigma_{h,k+s}-N(B_f)$. Now $\Sigma_g/f-N(B_f)$ is the quotient space of $\Sigma_{h,k+s}-N(B_f)$, where the quotient map is 2-fold covering when restricted on the preimage of each $\iota_l$ and is one-to-one elsewhere. Hence we shall have induced orientation on each $\iota_l$, $S_{q_i}$ and $S_{\delta_j}$. Now we reverse the order of this procedure and give the following two definitions:

But we can cut $\Sigma_g/f$ along a sequence of disjoint one-sided simple closed curves $\iota_l$, $l=1,\cdots,s$ (the tubular neighborhood of each
 $\iota_l$ is homeomorphic to M\"{o}bius band), such that the result surface, denoted by $\Sigma_{h,k+s}$ is orientable. We can give $\Sigma_{h,k+s}$ an arbitrary orientation, which induces orientation on each of the boundary components of $\Sigma_{h,k+s}-N(B_f)$. Now $\Sigma_g/f-N(B_f)$ is the quotient space of $\Sigma_{h,k+s}-N(B_f)$. Since this quotient map restricted on the preimage of each one-sided $\iota_l$ is a 2-fold covering and it is one-to-one elsewhere, we shall have well-defined induced orientation on each $\iota_l$, $S_{q_i}$ and $S_{\delta_j}$. Now we reverse the order of this procedure and give the following two definitions:
\begin{definition}
When $\Sigma_g/f$ is non-orientable, a sequence of disjoint, simple and one-sided loops $\iota_l\colon\{e^{i\theta}\}\rightarrow\Sigma_g/f-B_f$, $l=1,\cdots,s$ are said to be co-oriented if:
 \begin{enumerate}
  \item when cut $\Sigma_g/f$ along $\sqcup_{l=1}^s\iota_l$ we get an orientable surface $\Sigma_{h,k+s}$, and
  \item $\Sigma_{h,k+s}$ can be given an orientation such that the induced orientation on each of $\iota_l$ is the same as their induced orientation from $\{e^{i\theta}\}$.
 \end{enumerate}
\end{definition}

\begin{definition}
Given co-oriented one-sided loops $\{\iota_1,\cdots,\iota_s\}$ on $\Sigma_g/f$, where $s=0$ if $\Sigma_g/f$ is orientable. We can apply the map $\omega_f\colon\thinspace H_1(\Sigma_g/f-B_f)\rightarrow Z_n$ and obtain well-defined data $[h,n;(\cdots,\omega_f(S_{q_i}),\cdots);(\cdots,\omega_f(S_{\delta_j}),\cdots);(\cdots,\omega_f(\iota_l),\cdots)]$, which is called the data of $f$ based on co-oriented one-sided loops $\{\iota_1,\cdots,\iota_s\}$, or simply the data of $f$.
\end{definition}

The main theorem of this section is:
\begin{theorem}\label{only}
For any two orientation reversing periodic maps, if we choose a set of co-oriented one-sided loops for each of them and get exactly the same data, then these two maps must be conjugate to each other.
\end{theorem}

Before giving the proof of this theorem, we need to do much preparatory work.

\begin{lemma}\label{class2'}
When cut $\Sigma_g/f^2$ along $\Fix(\bar f)=\sqcup_{j=1}^k\sigma_j$, the result surface is connected if and only if $\Sigma_g/f$ is non-orientable.
\end{lemma}

\begin{proof}
If the result surface is connected, we choose an arbitrary point $x\in \Sigma_g/f^2-\Fix(\bar f)$ and a simple path $\alpha\colon\thinspace[0,1]\rightarrow \Sigma_g/f^2-\Fix(\bar f)$ from $x$ to $\bar f(x)$. Since $\bar f$ is orientation reversing, we are able to go around the loop $\pi_{\bar f}(\alpha)$ on $\Sigma_g/f$ and come back to $\pi_{\bar f}(x)$ with the opposite orientation. Hence $\Sigma_g/f$ is non-orientable.

If the result surface is disconnected, then $\bar f$ permutate its components. Hence the components number is two and each of the components is homeomorphic to $\Sigma_g/f$. Hence $\Sigma_g/f$ is orientable.
\end{proof}

For each one-sided loop $\iota_l$,  let $\sigma_{k+l}=\pi_{\bar f}^{-1}(\iota_l)$ be its lift on $\Sigma_g/f^2$. Following lemma~\ref{class2'}, we cut $\Sigma_g/f^2$ along $\Fix(\bar f)\sqcup\pi_{\bar f}^{-1}(\sqcup_{l=1}^s\iota_l)=\sqcup_{j=1}^{k+s}\sigma_j$ and get two oriented compact components denoted by $\Sigma_{h,k+s}^+$ and $\Sigma_{h,k+s}^-$. Each of these components is homeomorphic to the result surface obtained when we cut $\Sigma_g/f$ along $\sqcup_{l=1}^s\iota_l$. And they induce opposite orientations on their common boundaries, correspondingly on each $\iota_l$, $l=1,\cdots,s$. In the following we always suppose that the induced orientation of each loop $\iota_l$ from $\Sigma_{h,k+s}^+$ is the same as their given orientation.

\begin{lemma}\label{t_i}
The orbit length $t_j$ of $\sigma_j$, $j=1,\cdots,k+s$ is odd.
\end{lemma}
\begin{proof}
It's clear that $\pi_{f^2}\colon\thinspace\pi_{f^2}^{-1}(\sigma_j)\rightarrow\sigma_j$ is an $m$-fold covering, then we know that $\pi_{f^2}^{-1}(\sigma_j)$ is disjoint union of simple closed curves on $\Sigma_g$. From $\bar f(\sigma_j)=\sigma_j$, we know that for each circle $\Delta\subset \pi_{f^2}^{-1}(\sigma_j)$, $f(\Delta)\subset \pi_{f^2}^{-1}(\sigma_j)$. Thus $\pi_{f^2}^{-1}(\sigma_j)=\sqcup_{r=1}^{t_j} f^r(\Delta)$. Because the orbit space of these circles $f^r(\Delta), r=1,2,\cdots,t_j$ under the action of $f^2$ has only one component, we can deduce that $t_j$ must be an odd integer.
\end{proof}
\begin{lemma}\label{class1}
If $\Fix(\bar f)\neq\emptyset$ ($k>0$), then $m$ is odd.
\end{lemma}
\begin{proof}
Similar to the proof of the lemma ~\ref{t_i}, we suppose that $x\in\Fix(\bar f)$ and $\pi_{f^2}^{-1}(x)=\{\tilde{x}_1,\cdots,\tilde{x}_m\}$. The orbit space of $\{\tilde{x}_1,\cdots,\tilde{x}_m\}$ under the action of $f^2$ is $x\in\Sigma_g/f^2$, hence $m$ must be odd.
\end{proof}

\begin{theorem}\label{split}
When we cut $\Sigma_g$ along the circles in $\pi_{f^2}^{-1}(\sqcup_{j=1}^{k+s}\sigma_j)$, the result surface has exactly $2$ components.
\end{theorem}

\begin{proof}
We know that each component is a compact surface and $f$ induces a permutation among these components $\{\Sigma^{(\lambda)}\}$. Choose one of these components, denoted by $\Sigma_{\tau,e}$. Consider the set $\sqcup_{i=1}^d f^i(\Sigma_{\tau,e})$, with $d$ the smallest positive integer such that $f^d(\Sigma_{\tau,e})=\Sigma_{\tau,e}$.

Choose a circle $\Delta_j\subset\pi_{f^2}^{-1}(\sigma_j)$ and its tubular neighborhood $N(\Delta_j)=S^1\times (-1,1)$ in $\Sigma_g$. When we cut along $\Delta_j$, we get $S^1\times [0,1)$ and $S^1\times (-1,0]$, each of which is the neighborhood of a boundary of some component $\Sigma^{(\lambda)}$.

$f^{t_j}$ is a rotation when restricted on $\Delta_j$. By lemma \ref{t_i}, $t_j$ is odd and $f^{t_j}(S^1\times (-1,0])=S^1\times [0,1)$. Hence if $S^1\times [0,1)\subset\sqcup_{i=1}^d f^i(\Sigma_{\tau,e})$, we have $S^1\times (-1,0]=f^{t_j}(S^1\times [0,1))\subset\sqcup_{i=1}^d f^i(\Sigma_{\tau,e})$. This tells us that all the boundary circles of $\sqcup_{i=1}^d f^i(\Sigma_{\tau,e})$ can be paired with each other, such that if we glue each pair into one circle properly, we get a both open and closed subset of $\Sigma_g$, hence this subset is $\Sigma_g$ itself.

By the above analysis, in particular, we know that each component contains equal number of circles from the set of $2t_j$ circles, generated when we cut $\Sigma_g$ along $\pi_{f^2}^{-1}(\sigma_j)$. This implies that $d|2t_j$. Again, $t_j$ must be an odd integer, then:

If $d$ is an odd integer, we have $d|t_j$ and $f^{t_j}=(f^d)^{a_j}$ for some integer $a_j$. This implies that if $S^1\times [0,1)\subset\Sigma_{\tau,e}$, we have $S^1\times (-1,0]=(f^d)^{a_j}(S^1\times [0,1))\subset (f^d)^{a_j}(\Sigma_{\tau,e})=\Sigma_{\tau,e}$. And all the circle boundaries of $\Sigma_{\tau,e}$ can be paired with each other, such that when we glue each pair into one circle properly, we get a both open and closed subset of $\Sigma_g$ and we know that it is $\Sigma_g$ itself. This means that $d=1$. Then $\Sigma_{\tau,e}/f^2=\Sigma_{h,k+s}^+\sqcup\Sigma_{h,k+s}^-$, which contradicts to the continuation of quotient map $\pi_{f^2}$.

Hence $d$ is even and we have $(d/2)|t_j$. Then $f^{t_j}=(f^{d/2})^{a_j}$ for some odd integer $a_j$. This implies that if $S^1\times [0,1)\subset\Sigma_{\tau,e}$, we have $S^1\times (-1,0]=(f^{d/2})^{a_j}(S^1\times [0,1))\subset f^{d/2}(\Sigma_{\tau,e})$. In general, each circle boundary of $\Sigma_{\tau,e}$ can be paired with the one in $f^{d/2}(\Sigma_{\tau,e})$, such that when we glue each pair into one circle properly, we get a both open and closed subset of $\Sigma_g$ and we know that it is $\Sigma_g$ itself. This means that $d=2$.
\end{proof}

By the above theorem, we shall get two components $\Sigma_{\tau,e}^+$ and $\Sigma_{\tau,e}^-$ when we cut $\Sigma_g$ along $\pi_{f^2}^{-1}(\sqcup_{j=1}^{k+s}\sigma_j)$.

Now we have an orientation preserving period map $f^2_+\colon\thinspace\Sigma_{\tau,e}^+\rightarrow\Sigma_{\tau,e}^+$, a branched covering $\pi_{f^2_+}\colon\thinspace\Sigma_{\tau,e}^+\rightarrow\Sigma_{h,k+s}^+$ and a homomorphism $\omega_{f^2_+}\colon\thinspace H_1(\Sigma_{h,k+s}^+-B_{f^2_+})\rightarrow Z_m$. Since for each loop $\gamma\subset\Sigma_{h,k+s}^+$, $\omega_{f^2_+}([\gamma])=\omega_{f^2}([\gamma])$, for convenient, we always use the same notation $\omega_{f^2}$.

We denote the branch points of $\pi_{f^2_+}$ by $B_{f^2_+}=\{\tilde q_1,\cdots,\tilde q_b\}$. Now each boundary component $S_{\tilde q_i}$ or $\sigma_j$ of $\Sigma_{h,k+s}^+-N(B_{f^2}^+)$ has induced orientation.

\begin{definition}
Let $\theta_i=\omega_{f^2}(S_{\tilde q_i})$, $i=1,2,\cdots,b$; $\eta_j=\omega_{f^2}(\sigma_j)$, $j=1,2,\cdots,k+s$. Then we have a collection of data $[h,m;(\theta_1,\cdots,\theta_b);(\eta_1,\cdots,\eta_{k+s})]$, called the pre-data of $f$ based on co-oriented one-sided loops $\{\iota_1,\cdots,\iota_s\}$, or simply the pre-data of $f$.
\end{definition}

By results of section \ref{orientable}:
\begin{lemma}\label{predata}
The conjugacy class of $f^2_+\colon\thinspace\Sigma_{\tau,e}^+\rightarrow\Sigma_{\tau,e}^+$ is determined by the pre-data of $f$.
\end{lemma}

Similar to lemma \ref{index}, we have:
\begin{lemma}\label{type}
Suppose that for each $\sigma_j$, $\pi_{f^2}^{-1}(\sigma_j)$ has $t_j$ components and $f^{t_j}$ restricted on each of them has rotation number $p_j/q_j$. Then
 \begin{enumerate}
  \item $t_j=\gcd\{\omega_f(S_{\delta_j}),m\}$, $q_j=\frac{m}{t_j}$, $p_j=(\frac{\omega_f(S_{\delta_j})}{t_j})^{-1} \mod q_j$ for $j=1,\cdots,k$
  \item $t_{k+l}=\gcd\{\omega_f(\iota_l),n\}$, $q_{k+l}=\frac{n}{t_j}$, $p_{k+l}=(\frac{\omega_f({\iota_l})}{t_{k+l}})^{-1} \mod {q_{k+l}}$ for $l=1,\cdots,s$.
 \end{enumerate}
\end{lemma}

\begin{proof}
For $\sigma_j$, $j=1,\cdots,k$, we pick $x\in\sigma_j$ and choose $\tilde x\in\pi_{f^2}^{-1}(x)$. Then there are $m$ points in the orbit of $\tilde x$ under the action of $f$ and there are $m/t_j$ points in the orbit of $\tilde x$ under the action of $f^{t_j}$, thus $q_j=m/t_j$. We can take a tubular neighborhood of a loop $\Delta_j\subset\pi_{f^2}^{-1}(\sigma_j)$, homeomorphic to $S^1\times[-1,1]$. Then the map $f^{t_j}$ on this neighborhood is conjugate to
$$S^1\times[-1,1]\rightarrow S^1\times[-1,1]$$
$$(e^{i\theta},r)\mapsto(e^{i(\theta+2\pi\frac{p_j}{q_j})},-r).$$

The lift of $S_{\delta_j}$ on $S^1\times[-1,1]$ is the path $\{(e^{i\theta},1)|0\leq\theta\leq\frac{2\pi}{m/t_j}\}$ from $(e^{i0},1)$ to $(e^{ i\frac{2\pi}{q_j}},1)$. By lemma \ref{class1}, $m$ should be an odd integer, then $q_j$ is odd and $\gcd\{2p_j,q_j\}=1$. Since $f^{2t_j}$ has rotation number $2p_j/q_j$, suppose $u(2p_j)-vq_j=1$, we have $f^{2ut_j}(e^{i0},1)=(e^{i\frac{2\pi}{q_j}},1)$. By definition we have $\omega_f(S_{\delta_j})=2ut_j$. From $p_j(2ut_j)-vm=t_j$ we have $t_j=\gcd\{\omega_f(S_{\delta_j}),m\}$. From $(2u)p_j-vq_j=1$ we have $p_j=(\frac{\omega_f(S_{\delta_j})}{t_j})^{-1}\mod q_j$.

For $\sigma_{k+l}$, $l=1,\cdots,s$, $\iota_l$ has valency $[t_{k+l},q_{k+l},u]$ we pick $x\in\sigma_{k+l}$ choose $\tilde x\in\pi_{f^2}^{-1}(x)$. Then there are $n$ points in the orbit of $\tilde x$ under the action of $f$ and there are $n/t_{k+l}$ points in the orbit of $\tilde x$ under the action of $f^{t_{k+l}}$, thus $q_{k+l}=n/t_{k+l}$. We can choose a sufficiently small neighborhood of a circle $\Delta_{k+l}\subset\pi_{f^2}^{-1}(\sigma_{k+l})$, homeomorphic to $S^1\times[-1,1]$. Then the map $f^{t_{k+l}}$ on this neighborhood is conjugate to
$$S^1\times[-1,1]\rightarrow S^1\times[-1,1]$$
$$(e^{i\theta},r)\mapsto(e^{i(\theta+2\pi\frac{p_{k+l}}{q_{k+l}})},-r).$$
The lift of $\iota_l=\pi_{\bar f}(\sigma_{k+l})$ on $S^1\times[-1,1]$ is the path $\{(e^{i\theta},0)|0\leq\theta\leq \frac{2\pi}{q_{k+l}}\}$ from $(e^{i0},0)$ to $(e^{ i\frac{2\pi}{q_{k+l}}},0)$. Suppose $up_{k+l}-vq_{k+l}=1$, we have $f^{ut_{k+l}}((e^{i0},0))=(e^{i\frac{2\pi}{q_{k+l}}},0)$. By definition we have $\omega_f(\iota_l)=ut_{k+l}$. From $p_{k+l}(ut_{k+l})-vn=t_{k+l}$ we have $t_{k+l}=\gcd\{\omega_f(\iota_{l}),n\}$. From $up_{k+l}-vq_{k+l}=1$ we have $p_{k+l}=(\frac{\omega_f({\iota_l})}{t_{k+l}})^{-1} \mod q_{k+l}$.
\end{proof}

Based on the proof of the above lemma, we also have:
\begin{lemma}\label{data}
The data of $f$ and the pre-data of $f$ based on co-oriented one-sided loops $\{\iota_1,\cdots,\iota_s\}$ satisfy: $\omega_f(S_{q_i})=2\theta_i$, $i=1,2,\cdots,b$; $\omega_f(S_{\delta_j})=2\eta_j$, $j=1,2,\cdots,k$; $\omega_f(\iota_l)=\eta_{k+l} \mod m$, $l=1,2,\cdots,s$.
\end{lemma}

\begin{proof}[Proof of theorem \ref{only}]
Suppose that after cutting $\Sigma_g/f$ and $\Sigma_g'/f'$ along the chosen
two sets of co-oriented one-sided loops we obtain $\Sigma_{h,k+s}^+$, $\Sigma_{h,k+s}'^+$ and the same data $[h,2m;(2\theta_1,\cdots,2\theta_b);(2\eta_1,\cdots,2\eta_k);(\zeta_1,\cdots,\zeta_s)]$.

By theorem \ref{split}, we have $f^2_+\colon\thinspace\Sigma_{\tau,e}^+\rightarrow\Sigma_{\tau,e}^+$ and $f'^2_+\colon\thinspace\Sigma_{\tau,e}'^+\rightarrow\Sigma_{\tau,e}'^+$. Because $f$ and $f'$ have the same data, by lemma \ref{data}, $f^2_+$ and $f'^2_+$ also have the same data. By lemma \ref{predata}, there exists an orientation preserving homeomorphism $h_+\colon\thinspace\Sigma_{\tau,e}^+\rightarrow\Sigma_{\tau,e}'^+$, such that $f'^2_+=h_+\circ f^2_+\circ h_+^{-1}$. Then there is an induced homeomorphism $\bar h_+\colon\thinspace\Sigma_{h,k+s}^+\rightarrow\Sigma_{h,k+s}'^+$, such that $\bar h_+\circ\pi_{f^2_+}=\pi_{f'^2_+}\circ h_+$.

For each boundary component $\sigma_j\subset\Sigma_{h,k+s}^+$, $\bar h_+(\sigma_j)$ is the boundary component of $\Sigma_{h,k+s}'^+$.
By lemma \ref{type}, we can find out the components number $t_j$ of $\pi_{f^2}^{-1}(\sigma_j)$ and the rotation number $p_j/q_j$ of the map $f^{t_j}$ restricted on a loop $\Delta_j\subset\pi_{f^2}^{-1}(\sigma_j)$. Then the map $\pi_{f^2}\colon\thinspace\Delta_j\rightarrow\sigma_j$ is an $m/t_j$-fold covering.

For convenient, we give these boundary loops $\Delta_j$, $h_+(\Delta_j)$, $\sigma_j$ and $\bar h_+(\sigma_j)$ some nice parameterization such that $f^{t_j}(e^{i\theta})=e^{i(\theta+2\pi\frac{p_j}{q_j})}$, $\pi_{f^2}(e^{i\theta})=e^{i (m/t_j)\theta}$; $f'^{t_j}(e^{i\theta})=e^{i(\theta+2\pi\frac{p_j}{q_j})}$, $\pi_{f'^2}(e^{i\theta})=e^{i (m/t_j)\theta}$.

Then we can make an isotopy of $\bar h_+$ on the tubular neighborhood of $\sigma_j$ on $\Sigma_{h,k+s}^+$ such that $\bar h_+(e^{i\theta})=e^{i\theta}$ when restricted on $\sigma_j$. The lift of this isotopy is an isotopy of $h_+$. Notice that $f'^2_+=h_+\circ f^2_+\circ h_+^{-1}$ still holds after the isotopy. Now suppose that $h_+(e^{i\theta})=e^{i\phi(\theta)}$ when restricted on $\Delta_j$, then from $\pi_{f'^2}\circ h_+=\bar h_+\circ\pi_{f^2}$ we have $e^{i (m/t_j)\phi(\theta)}=e^{i (m/t_j)\theta}$. Hence $\phi(\theta)=\theta+2\pi\frac{r}{m/t_j}$ for some integer $r$. It's easy to verify that $f'^{t_i}=h_+\circ f^{t_i}\circ h_+^{-1}$ when restricted on $h_+(\Delta_j)$. Plus that $f=(f^2)^{\frac{1-t_j}{2}}\circ f^{t_j}$, $f'=(f'^2)^{\frac{1-t_j}{2}}\circ f'^{t_j}$ and $f'^2_+=h_+\circ f^2_+\circ h_+^{-1}$ we have $f'=h_+\circ f\circ h_+^{-1}$ when restricted on the boundaries of $\Sigma_{\tau,e}'^-$.

When $x\in\Sigma_{\tau,e}^-$, $f^{-1}(x)\in \Sigma_{\tau,e}^+$. Then $h_+$ induces an orientation preserving homeomorphism
$$h_-\colon\thinspace\Sigma_{\tau,e}^-\rightarrow\Sigma_{\tau,e}'^-$$ $$x\mapsto f'\circ h_+\circ f^{-1}(x).$$

Then for $x\in\pi_{f^2}^{-1}(\sqcup_{j=1}^{k+s}\sigma_j)$, $h_-(x)=f'\circ h_+\circ f^{-1}(x)=h_+(x)$. Then we obtain a well-defined orientation preserving
homeomorphism $$h\colon\thinspace\Sigma_g\rightarrow\Sigma_\tau$$ $$x\mapsto h_+(x)$$ when $x\in\Sigma_{\tau,e}^+$; $$x\mapsto h_-(x)$$ when $x\in\Sigma_{\tau,e}^-$.

It's easy to verify that $f'=h\circ f\circ h^{-1}$.
\end{proof}

\begin{theorem}\label{reconstruct}
There is an orientation reversing periodic map of period $n=2m$ on closed orientable surface of genus $g$ with data $[h,2m;(2\theta_1,\cdots,2\theta_b);(2\eta_1,\cdots,2\eta_k);(\zeta_1,\cdots,\zeta_s)]$ if and only if the following conditions are satisfied:
 \begin{enumerate}
 \item $k+s>0$
 \item When $m$ even, $k=0$
 \item $\zeta_l$ is odd for $l=1,2,\cdots,s$
 \item $[h,m;(\theta_1,\cdots, \theta_b);(\eta_1,\cdots,\eta_k,\zeta_1,\cdots,\zeta_s)]$ is the data of an orientation preserving periodic map
 \item $g-1=m\left(2h-2+k+s+\Sigma_{i=1}^b(1-\frac{\gcd\{\theta_i,m\}}{m})\right)$.
 \end{enumerate}
\end{theorem}

\begin{proof}
The necessity of these conditions is easy to see. Here we only look at the sufficiency.

Suppose that the given data satisfy each condition listed above. And $\psi\colon\thinspace\Sigma_{\tau,e}\rightarrow\Sigma_{\tau,e}$ is an orientation preserving periodic map with data $[h,m;(\theta_1,\cdots, \theta_b);(\eta_1,\cdots,\eta_k,\zeta_1,\cdots,\zeta_s)]$.

Let $t_j$, $p_j$ and $q_j$, $j=1,2,\cdots,k+s$ be given by formulas in lemma \ref{type}. When $k>0$, $m$ is odd and $t_j=\gcd\{2\eta_j,m\}$ is odd. Since $\zeta_l$ is odd, we have $t_{k+l}=\gcd\{\zeta_l,n\}$ to be odd.

For each boundary component $\sigma_j$ of $\Sigma_{h,k+s}=\Sigma_{\tau,e}/\psi$, by lemma \ref{index}, there are $t_j$ components in $\pi_{\psi}^{-1}(\sigma_j)$. Then choose one of them, say $\Delta_j$, we have  $\pi_{\psi}^{-1}(\sigma_j)=\sqcup_{i=1}^{t_j}\psi^i(\Delta_j)$. Then there exists a parametrization $h_j\colon\thinspace S^1=\{e^{i\theta}\}\rightarrow\Delta_j$, compatible with the induced orientation of $\Delta_j$ as boundary of $\Sigma_{\tau,e}$, such that
$$h_i^{-1}\circ\psi^{t_j}\circ h_i\colon\thinspace S^1\rightarrow S^1$$
$$e^{i\theta}\mapsto e^{i(\theta+2\pi\frac{2p_j}{q_j})}$$
when $j=1,2,\cdots,k+s$.

Then we have parametrization $h_j^i\colon\thinspace S^1\rightarrow\psi^i(\Delta_j)$ with $h_j^i=\psi^i\circ h_j$.

Now we let $\varphi_j$ be the map
$$\varphi_j\colon\thinspace S^1\rightarrow S^1$$
$$e^{i\theta}\mapsto e^{i(\theta+2\pi\frac{p_j}{q_j})}$$
when $j=1,2,\cdots,k+s$.

Then we know that $\psi^{t_j}=h_j^i\circ\varphi_j^2\circ (h_j^i)^{-1}$, when restricted on $\psi^i(\Delta_j)$.

We construct a map: $$\Phi\colon\thinspace\pi_{\psi}^{-1}(\sqcup_{j=1}^{k+s}\sigma_j)\rightarrow\pi_{\psi}^{-1}(\sqcup_{j=1}^{k+s}\sigma_j)$$
$$x\mapsto\psi^{\frac{1-t_j}{2}}\circ h_j^i\circ\varphi_j\circ (h_j^i)^{-1}(x)$$ when $x\in\psi^i(\Delta_j)$, $j=1,2,\cdots,k+s$, $i=1,2,\cdots,t_j$.

Then for $x\in\psi^i(\Delta_j)$, $\Phi(x)\in\psi^{\frac{1-t_j}{2}+i}(\Delta_i)$ and we have $\Phi^2(x)=\psi^{\frac{1-t_j}{2}}\circ
h_j^{\frac{1-t_j}{2}+i}\circ\varphi_j\circ (h_j^{\frac{1-t_j}{2}+i})^{-1}\circ\psi^{\frac{1-t_j}{2}}\circ h_j^i\circ\varphi_j\circ
(h_j^i)^{-1}(x)=\psi^{\frac{1-t_j}{2}}\circ h_j^{\frac{1-t_j}{2}+i}\circ\varphi_j^2\circ (h_j^i)^{-1}(x)=\psi^{1-t_j}\circ h_j^i\circ\varphi_j^2\circ
(h_j^i)^{-1}(x)=\psi(x)$.

Make two copies of $\Sigma_{\tau,e}$, denoted by $\Sigma_{\tau,e}^+$ and $\Sigma_{\tau,e}^-$. The points of these two copies corresponding to $x\in\Sigma_{\tau,e}$ are
denoted by $(x,+)$ and $(x,-)$ respectively. Their boundaries are $(\pi_{\psi}^{-1}(\sqcup_{i=1}^{k+s}\sigma_i),+)$ and
$(\pi_{\psi}^{-1}(\sqcup_{i=1}^{k+s}\sigma_i),-)$ respectively.

Then we obtain a new surface $\Sigma_{\tau,e}^+\sqcup_{\Phi}\Sigma_{\tau,e}^-$ by gluing $(x,-)$ with $(\Phi(x),+)$ for
$x\in\pi_{\psi}^{-1}(\sqcup_{i=1}^{k+s}\sigma_i)$. We give this new closed surface an orientation such that the map
$\Sigma_{\tau,e}\rightarrow\Sigma_{\tau,e}^+\sqcup_{\Phi}\Sigma_{\tau,e}^-$ giving by $x\mapsto(x,+)$ is orientation preserving.

Now we construct a periodic map $f$ on it: $$f\colon\thinspace\Sigma_{\tau,e}^+\sqcup_{\Phi}\Sigma_{\tau,e}^-\rightarrow\Sigma_{\tau,e}^+\sqcup_{\Phi}\Sigma_{\tau,e}^-$$ $$(x,+)\mapsto
(x,-)$$ for $(x,+)$ in the interior of $\Sigma_{\tau,e}^+$; $$(x,-)\mapsto (\psi(x),+)$$ for $(x,-)$ in the interior of $\Sigma_{\tau,e}^-$;
$$\{(x,-),(\Phi(x),+)\}\mapsto \{(\Phi(x),-),(\Phi^2(x),+)\}$$ otherwise.

It's easy to verify that $f$ is a well-defined orientation reversing periodic map with the given data.
\end{proof}

\section{The equivalent relations of the data}\label{equivalent relation}
Given data $[h,2m;(2\theta_1,\cdots,2\theta_b);(2\eta_1,\cdots,2\eta_k);(\zeta_1,\cdots,\zeta_s)]$ satisfying the conditions listed in the theorem \ref{reconstruct}.

Let $d=1$ if $h>0$ and $d=\gcd\{\theta_1,\cdots,\theta_b,\eta_1,\cdots,\eta_k\}$ if $h=0$.

Suppose that $\tau$, $\rho$ and $\sigma$ are permutations of $\{1,2,\cdots,b\}$, $\{1,2,\cdots,k\}$ and $\{1,2,\cdots,s\}$.

Now we define the following equivalent relations:
\begin{tabbing}
(R0)\= $[h,2m;(2\theta_1,\cdots,2\theta_b);(2\eta_1,\cdots,2\eta_k);(\zeta_1,\cdots,\zeta_s)]$ is equivalent with\kill
(R0)\> $[h,2m;(2\theta_1,\cdots,2\theta_b);(2\eta_1,\cdots,2\eta_k);(\zeta_1,\cdots,\zeta_s)]$ is equivalent with\\
    \> $[h,2m;(2\theta_{\tau(1)},\cdots,2\theta_{\tau(b)});(2\eta_{\rho(1)},\cdots,2\eta_{\rho(k)});(\zeta_{\sigma(1)},\cdots,\zeta_{\sigma(s)})]$\\
(R1)\> $[h,2m;(2\theta_1,\cdots,2\theta_b);(2\eta_1,\cdots,2\eta_k);(\zeta_1,\cdots,\zeta_s)]$ is equivalent with\\
    \> $[h,2m;(-2\theta_1,\cdots,-2\theta_b);(-2\eta_1,\cdots,-2\eta_k);(-\zeta_1,\cdots,-\zeta_s)]$\\
(R2)\> $[h,2m;(\cdots,2\theta_i,\cdots);(\cdots,2\eta_j,\cdots);(\cdots,\zeta_l,\cdots,\zeta_s)]$ ($s>0$) is equivalent with\\
    \> $[h,2m;(\cdots,-2\theta_i,\cdots);(\cdots,2\eta_j,\cdots);(\cdots,\zeta_l,\cdots,\zeta_s+2\theta_i)]$, $1\leq i\leq b$ and\\
    \> $[h,2m;(\cdots,2\theta_i,\cdots);(\cdots,-2\eta_j,\cdots);(\cdots,\zeta_l,\cdots,\zeta_s+2\eta_j)]$, $1\leq j\leq k$ and\\
    \> $[h,2m;(\cdots,2\theta_i,\cdots);(\cdots,2\eta_j,\cdots);(\cdots,-\zeta_l,\cdots,\zeta_s+2\zeta_l)]$, $1\leq l\leq s-1$\\
(R3)\> $[h,2m;(\cdots);(\cdots);(\cdots,\zeta_{s-2},\zeta_{s-1},\zeta_s)]$ ($s\geq 3$) is equivalent with\\
    \> $[h+1,2m;(\cdots);(\cdots);(\cdots,\zeta_{s-2}+\zeta_{s-1}+\zeta_s)]$\\
(R4)\> $[h,2m;(\cdots);(\cdots);(\zeta_1,\zeta_2)]$ is equivalent with\\
    \> $[h,2m;(\cdots);(\cdots);(\zeta_1+2d,\zeta_2-2d)]$
\end{tabbing}

The main result of this section is:
\begin{theorem}\label{equivalent}
Two collections of data for the orientation reversing periodic maps determine the same conjugacy class if and only if they are equivalent under the relations (R0)$\sim$(R4).
\end{theorem}

The proof of the above theorem is lengthy, hence we decompose it into the the following theorems \ref{s=0}, \ref{s=1} and \ref{s=2}.

In this section we keep using the notations of section \ref{new data}:

Let $f$ be the orientation reversing periodic map on $\Sigma_g$ which is a representative of the conjugacy class determined by data $[h,2m;(2\theta_1,\cdots,2\theta_b);(2\eta_1,\cdots,2\eta_k);(\zeta_1,\cdots,\zeta_s)]$. Let $\iota_l\subset\Sigma_g/f-B_f$, $l=1,\cdots,s$ be the corresponding co-oriented one-sided loops and when we cut $\Sigma_g/f$ along $\sqcup_{l=1}^s\iota_l$ we get oriented surface $\Sigma_{h,k+s}^+$.

\begin{lemma}\label{R01}
Two collections of data which are equivalent under (R0) and (R1) determine the same conjugacy class.
\end{lemma}
\begin{proof}
The order of branch points, boundary components and co-oriented one-sided loops do not really matter. And when we reverse the orientations of each $\iota_l$ we get a new set of co-oriented one-sided loops which change the sign of $\omega_f$ value in the data.
\end{proof}

In particular, when $s=0$, we have:
\begin{theorem}\label{s=0}
Given two orientation reversing periodic maps $f$ on $\Sigma_g$ and $f'$ on $\Sigma_g'$,
if $\Sigma_g/f$ and $\Sigma_g'/f'$ are both orientable, then $f$ and $f'$ are conjugate to each other if and only if they
have data which are equivalent under (R0) and (R1).
\end{theorem}

When $\Sigma_g/f$ is non-orientable, the choosing of one-sided loops is kind of arbitrary. Therefore, the cases of $s>0$ are a little complicated. In order to prove that the later collection of data in R(2), R(3) and R(4) determine the same conjugacy class as the former collection of data, we need to prove that we can choose the co-oriented one-sided loops on $\Sigma_g/f-B_f$ differently to obtain the later collection of data.

\begin{lemma}\label{R2}
Two collections of data which are equivalent under the relation (R2) determine the same conjugacy class.
\end{lemma}

\begin{proof}
For boundary component $S_{q_i}$ of $\Sigma_g/f-N(B_f)$, let $\gamma$ be the loop which is parallel with $S_{q_i}$ and $\omega_f(\gamma)=\omega_f(S_{q_i})=2\theta_i$. Join the point $x=\iota_s(0)$ to a point $y\in \gamma(0)$ by a simple path $\alpha\subset\Sigma_g/f-N(B_f)-\sqcup_{l=1}^{s-1}\iota_l$.

A small isotopy of the loop $\iota_s\cdot\alpha\cdot\gamma\cdot\bar{\alpha}$ supported on the tubular neighborhood of $\alpha$ would give us a new simple one-sided loop $\iota_s'$. It's easy to see that $\{\iota_1,\cdots,\iota_{s-1},\iota_s'\}$ is co-oriented and induces the same orientations as $\{\iota_1,\cdots,\iota_{s-1},\iota_s\}$ does on all boundary components of $\Sigma_g/f-N(B_f)$ except for $S_{q_i}$. Hence $\omega_f$ would change sign at $S_{q_i}$ and $\omega_f(\iota_s')=\omega_f(\iota_s)+\omega_f(\gamma)$, which gives us the later collection of data in the first case of (R2).

The second case of (R2) is quite the same as the first case. The third case of (R2) is similar, we only need to set $\gamma$ to be the boundary of a tubular neighborhood of $\iota_l$, $1\leq l\leq s-1$. Then we isotope $\iota_s\cdot\alpha\cdot\gamma\cdot\bar{\alpha}$ in the tubular neighborhood of $\alpha$ and get a new one-side simple loop $\iota_s'$. Then $\{\iota_1,\cdots,\bar{\iota_l},\cdots,\iota_{s-1},\iota_s'\}$ would be co-oriented. Since $\gamma$ goes around $\iota_l$ twice, we have $\omega_f(\iota_s')=\omega_f(\iota_s)+2\omega_f(\iota_l)$. Then we have the later collection of data in the third case of (R2).
\end{proof}

\begin{lemma}\label{balance}
If two collections of data $[h,2m;(2\theta_1,\cdots,2\theta_b);(2\eta_1,\cdots,2\eta_k);(\zeta_1^{(i)},\cdots,\zeta_{s_i}^{(i)})]$, $i=1,2$ determine the same conjugacy class, then $\zeta_1^{(1)}+\cdots+\zeta_{s_1}^{(1)}=\zeta_1^{(2)}+\cdots+\zeta_{s_2}^{(2)} \mod n$.
\end{lemma}
\begin{proof}
Suppose that $f$ is the corresponding periodic map of these two collections of data. Let $\sigma_{k+1}^{(i)},\cdots,\sigma_{k+s_i}^{(i)}$, $i=1,2$ on $\Sigma_g/f^2$ be the lift of two sequences of co-oriented one-sided loops corresponding to these two collections of data. We can make an isotopy such that $\sigma_{k+1}^{(1)}\sqcup\cdots\sqcup\sigma_{k+s_1}^{(1)}$ and $\sigma_{k+1}^{(2)}\sqcup\cdots\sqcup\sigma_{k+s_2}^{(2)}$ have finite intersection $\{y_1,\cdots,y_r\}$. Now we cut the surface $\Sigma_g/f^2$ along $\Fix(\bar f)=\sqcup_{j=1}^k\sigma_j$ and obtain a compact surface $\Sigma_{h',2k}$ with induced orientation on boundaries $\sigma_j^{+,-}$, $j=1,2,\cdots,k$.

Suppose that $\sigma_j^+$ induces orientation on boundary component $\delta_j$ of $\Sigma_g/f$ such that $\omega_f(S_{\delta_j})=2\eta_j$ contributes to the same part $(2\eta_1,\cdots,2\eta_k)$ of these two collections of data, then  $\sigma_j^+\subset\Sigma_{h^{(1)},k+s_1}^+\cap\Sigma_{h^{(2)},k+s_2}^+$ and $\sigma_j^-\subset\Sigma_{h^{(1)},k+s_1}^-\cap\Sigma_{h^{(2)},k+s_2}^-$. Hence $\sigma_j^{+,-}$, $j=1,\cdots,k$ has no intersection with $\Sigma_{h^{(1)},k+s_1}^+\cap\Sigma_{h^{(2)},k+s_2}^-$. Similarly, $B_{f^2}$ has no intersection with $\Sigma_{h^{(1)},k+s_1}^+\cap\Sigma_{h^{(2)},k+s_2}^-$.

Since $\bar f(\partial(\Sigma_{h^{(1)},k+s_1}^+\cap\Sigma_{h^{(2)},k+s_2}^-))=\partial(\Sigma_{h^{(1)},k+s_1}^-\cap\Sigma_{h^{(2)},k+s_2}^+)$ and $(\Sigma_{h^{(1)},k+s_1}^+\cap\Sigma_{h^{(2)},k+s_2}^-)\cap(\Sigma_{h^{(1)},k+s_1}^-\cap\Sigma_{h^{(2)},k+s_2}^+)=\pi_{\bar f}^{-1}(\sigma_{k+1}^{(1)}\sqcup\cdots\sqcup\sigma_{k+s_1}^{(1)})\cap\pi_{\bar f}^{-1}(\sigma_{k+1}^{(2)}\sqcup\cdots\sqcup\sigma_{k+s_2}^{(2)})=\pi_{\bar f}^{-1}(\{y_1,\cdots,y_r\})$, it's easy to verify that $\pi_{\bar f}\colon\thinspace \Sigma_{h^{(1)},k+s_1}^+\cap\Sigma_{h^{(2)},k+s_2}^-\rightarrow\Sigma_g/f-B_f$ is one-to-one on $\Sigma_{h^{(1)},k+s_1}^+\cap\Sigma_{h^{(2)},k+s_2}^--\pi_{\bar f}^{-1}(\{y_1,\cdots,y_r\})$ and is two-to-one on $\pi_{\bar f}^{-1}(\{y_1,\cdots,y_r\})$. Then we know that $K=\pi_{\bar f}(\Sigma_{h^{(1)},k+s_1}^+\cap\Sigma_{h^{(2)},k+s_2}^-)$ is a subcomplex of $\Sigma_g/f-B_f$, with boundary the co-oriented one-sided loops. Give 2-simplices in $K$ the induced orientation from $\Sigma_{h^{(1)},k+s_1}^+\cap\Sigma_{h^{(2)},k+s_2}^-$, we have $\zeta_1^{(1)}+\cdots+\zeta_{s_1}^{(1)}-\zeta_1^{(2)}-\cdots-\zeta_{s_2}^{(2)}=\omega_f(\partial K)=0$.
\end{proof}

\begin{theorem}\label{s=1}
Given two orientation reversing periodic maps $f$ on $\Sigma_g$ and $f'$ on $\Sigma_g'$, if $\Sigma_g/f$ and $\Sigma_g'/f'$ are both non-orientable and both have data with $s=1$, then $f$ and $f'$ are conjugate to each other if and only if the corresponding data are equivalent under (R0) and (R2).
\end{theorem}

\begin{proof}
By lemma~\ref{R01} and lemma~\ref{R2}, we know that two collections of data which are equivalent under (R0) and (R2) determine the same conjugacy class. Now we suppose that $f$ on $\Sigma_g$ and $f'$ on $\Sigma_g'$ are conjugate to each other, which means there exists orientation preserving  $h\colon\thinspace \Sigma_g\rightarrow\Sigma_g'$ such that $f'=h\circ f\circ h^{-1}$. Then we have an induced homeomorphism $\bar h\colon\thinspace\Sigma_g/f\rightarrow\Sigma_g'/f'$. And for any $q_i\in B_f$, we have $h(q_i)\in B_{f'}$ and $\omega_f(S_{q_i})=\pm\omega_{f'}(S_{h(q_i)})$, here the sign is determined by the choice of co-oriented one-sided loops. Similarly, we have $\omega_f(S_{\delta_j})=\pm\omega_{f'}(S_{h(\delta_j)})$ for $\delta_j\subset\partial(\Sigma_g/f)$. Hence if $f$ has data $[h,2m;(\cdots,2\theta_i,\cdots);(\cdots,2\eta_j,\cdots);(\zeta_{k+1})]$, then the data of $f'$ is equivalent to $[h,2m;(\cdots,2\theta_i,\cdots);(\cdots,2\eta_j,\cdots);(\zeta_{k+1}')]$ under (R0) and (R2). Notice that $[h,2m;(\cdots,2\theta_i,\cdots);(\cdots,2\eta_j,\cdots);(\zeta_{k+1})]$ is also the data of $f'$ based on co-oriented one-sided loop $\{\bar h(\iota_1)\}$. By lemma \ref{balance}, we have $\zeta_{k+1}'=\zeta_{k+1} \mod n$.
\end{proof}

\begin{lemma}\label{R3}
Two collections of data which are equivalent under the relation (R3) determine the same conjugacy class.
\end{lemma}
\begin{proof}
In this case we have $s\geq 3$ co-oriented one-sided loops. By the knowledge of elementary topology, we can replace three of them by one one-sided loop and one handle body.

After the operation above, the data $[h,2m;(\cdots,2\theta_i,\cdots);(\cdots,2\eta_j,\cdots);(\cdots,\zeta_{s-2},\zeta_{s-1},\zeta_s)]$ becomes $[h+1,2m;(\cdots,2\theta_i',\cdots);(\cdots,2\eta_j',\cdots);(\cdots,\zeta_{s-2}')]$, here $\theta_i'=\pm\theta_i$, $i=1,2,\cdots,b$; $\eta_j'=\pm\eta_j$, $j=1,2,\cdots,k$; $\zeta_l'=\pm\zeta_l$, $l=1,2,\cdots,s-3$. By relation (R2), these data determine the same conjugacy class with $[h+1,2m;(\cdots,2\theta_i,\cdots);(\cdots,2\eta_j,\cdots);(\cdots,\zeta_{s-3},\zeta_{s-2}'')]$. By lemma \ref{balance},  $\zeta_{s-2}''= \zeta_{s-2}+\zeta_{s-1}+\zeta_s \mod n$.
\end{proof}

\begin{lemma}\label{operation2}
Suppose that $\{\iota_1,\iota_2\}$ is a set of co-oriented loops. Then for any simple loop $\gamma\subset\Sigma_g/f-B_f-\iota_1-\iota_2$ which does not separate $\iota_1$ and $\iota_2$, we can find a new simple one-sided loop $\iota_1'\subset\Sigma_g/f-B_f$ such that $\omega_f(\iota_1')=\omega_f(\iota_1)+\omega_f(\gamma)$.
\end{lemma}

\begin{proof}
Join the point $\iota_1(0)$ to the point $\gamma(0)$ by a simple path $\alpha\subset\Sigma_g/f-B_f-\iota_1-\iota_2$. If $\iota_1\cdot\alpha\cdot\gamma\cdot\bar{\alpha}$ is isotopy to a simple loop $\iota_1'$ on $\Sigma_g/f-N(B_f)$, we have $\omega_f(\iota_1')=\omega_f(\iota_1)+\omega_f(\gamma)$. Otherwise, since $\gamma$ does not separate $\iota_1$ and $\iota_2$, we can join $\iota_1(0)$ to $\gamma(0)$ by simple path $\alpha\subset\Sigma_g/f-B_f$ which intersects $\iota_2$ exactly once. Then $\iota_1\cdot\alpha\cdot\gamma\cdot\bar{\alpha}$ is isotopy to a simple one-sided loop $\iota_1'$ on $\Sigma_g/f-B_f$ and $\omega_f(\iota_1')=\omega_f(\iota_1)+\omega_f(\gamma)$.
\end{proof}

\begin{lemma}\label{operation3}
Suppose that $\{\iota_1,\iota_2\}$ is a set of co-oriented loops. We cut $\Sigma_g/f$ along $\iota_1$, $\iota_2$ and obtain $\Sigma_{h,k+2}^+$. Then if $h>0$, we can find a simple loop $\gamma$ on $\Sigma_g/f-B_f-\iota_1-\iota_2$, which satisfies: (1)$\gamma$ does not separate $\iota_1$ and $\iota_2$; (2)$\omega_f(\gamma)=2$. If $h=0$, we can find several disjointed simple loops $\gamma_i$, $i=1,2,\cdots,r$ on $\Sigma_g/f-B_f-\iota_1-\iota_2$, which satisfies: (1)$\sqcup_{i=1}^{r}S_i$ does not separate $\iota_1$ and $\iota_2$; (2)$\sum_{i=1}^{r}\omega_f(\gamma_i)=2\gcd\{\theta_1,\cdots,\theta_b,\eta_1,\cdots,\eta_k\}$.
\end{lemma}

\begin{proof}
Case 1: $h>0$

We can find a simple loop $l$ on $\Sigma_{h,k+2}^+$, such that $l$ does not separate $\pi_{\bar f}^{-1}(\iota_1)$, $\pi_{\bar f}^{-1}(\iota_2)$. Then we cut $\Sigma_{h,k+2}^+$ along $l$ and obtain a new surface $\Sigma_{h-1,k+4}$. Denote the two new boundaries with induced orientation by $l^+$, $l^-$.

Because $\theta_1+\cdots+\theta_b+\eta_1+\cdots+\eta_k+\zeta_1+\zeta_2=\theta_1+\cdots+\theta_b+\eta_1+\cdots+\eta_k+\zeta_1+\zeta_2+1-1$,
by theorem \ref{Nielsen12}, there exists a periodic map $\varphi$ on $\Sigma_{\tau-1,e+2}$ with data $[h-1,m;(\theta_1,\cdots,\theta_b);(\eta_1,\cdots,\eta_k,\zeta_1,\zeta_2,1,-1)]$. And the branch covering is $\pi_{\varphi}\colon\thinspace\Sigma_{\tau-1,e+2}\rightarrow\Sigma_{h-1,k+4}$. Here $\omega_{\varphi}(l^+)=1$

By lemma \ref{index}, the components number of $\pi_{\varphi}^{-1}(l^+)$ and $\pi_{\varphi}^{-1}(l^-)$ are both one. If we glue these two boundaries of $\Sigma_{\tau-1,e+2}$ properly we can obtain a new surface $\Sigma_{\tau,e}$ with an induced periodic map $\varphi'$ on it with data $[h,m;(\theta_1,\cdots,\theta_b);(\eta_1,\cdots,\eta_k,\zeta_1,\zeta_2)]$. And the branch covering is $\pi_{\varphi'}\colon\thinspace\Sigma_{\tau,e}\rightarrow\Sigma_{h,k+2}$.

By lemma \ref{predata}, there is an orientation preserving homeomorphism $h\colon\thinspace\Sigma_{\tau,e}\rightarrow\Sigma_{\tau,e}^+$ such that $f^2_+=h\circ\varphi'\circ h^{-1}$. Then there is an induced homeomorphism $\bar h\colon\thinspace \Sigma_{h,k+2}\rightarrow\Sigma_{h,k+2}^+$ such that $\pi_{f^2_+}\circ h=\bar h\circ\pi_{\varphi'}$. Hence $\gamma=\bar h(l^+)$ is the required simple loop with $\omega_f(\gamma)=2\omega_{\varphi}(l^+)=2$.

Case 2: $h=0$

Let $\gcd\{\theta_1,\cdots,\theta_b,\eta_1,\cdots,\eta_k\}=u_1\theta_1+\cdots+u_b\theta_b+v_1\eta_1+\cdots+v_k\eta_k$ for integers
$u_1,\cdots,u_b,v_1,\cdots,v_k$. Then we can find $|u_i|$ circles around $q_i$ for each $i=1,2,\cdots,b$ with proper orientation and $|v_i|$ circles parallel with $\sigma_j$ for each $j=1,2,\cdots,k$ with proper orientation. It's easy to see that they are the required simple loops.
\end{proof}

\begin{lemma}\label{R4}
Two collections of data which are equivalent under the relation (R4) determine the same conjugacy class.
\end{lemma}

\begin{proof}
Suppose that $\{\iota_1,\iota_2\}$ is a set of co-oriented loops which gives us the former collection of data in (R4). We cut $\Sigma_g/f$ along $\iota_1$, $\iota_2$ and obtain $\Sigma_{h,k+2}^+$.

If $h>0$ and $d=1$, by lemma \ref{operation2} and lemma \ref{operation3}, we can find a new simple one-sided loop $\iota_1'$ with $\omega_f(\iota_1')=\omega_f(\iota_1)+2$. By the knowledge of topology of non-orientable surface, we can always find another simple one-sided loop $\iota_2'$ such that $\{\iota_1',\iota_2'\}$ is co-oriented. And data based on $\{\iota_1',\iota_2'\}$ is equivalent with the later collection of data in (R4) under relation (R2).

If $h=0$ and $d=\gcd\{\theta_1,\cdots,\theta_b,\eta_1,\cdots,\eta_k\}$, we only need to repeat the above process finite times to get the new set of co-oriented one-sided loops.
\end{proof}

\begin{lemma}\label{Hang-1}
Suppose that $\{\iota_1,\iota_2\}$ is a set of co-oriented loops. Then for an arbitrary one-sided loop $\iota\subset\Sigma_g/f-B_f$, we always have $\omega_f(\iota)=\pm \omega_f(\iota_1)+2rd$ or $\omega_f(\iota)=\pm \omega_f(\iota_2)+2rd$ for some integer $r$.
\end{lemma}

\begin{proof}
We can suppose that $\Sigma_g/f$ is triangulated with $\iota_1$ and $\iota_2$ to be sub-complexes. Make a small isotopy, such that $\iota$ is properly embedded. Then we can give each 2-simplex an orientation such that they are compatible everywhere except along $\iota_1$ and $\iota_2$. Since each time $\iota$ passes through $\iota_1$ or $\iota_2$ the orientation of nearby 2-simplexes would change once and $\iota$ has non-orientable neighborhood, we know that $\iota$ passes through $\iota_1\sqcup\iota_2$ odd times. In other words, $\iota_1$ and $\iota_2$ divide $\iota$ into odd pathes.

Then one of these pathes, denoted by $\alpha_1$, must have end points both in, say $\iota_1$. Suppose that the end points of $\alpha_1$ divide $\iota_1$ into two pathes $\beta_1$ and $\beta_2$, with $\beta_1(0)=\bar{\beta_2}(0)=\alpha_1(1)$. And suppose that $\iota=\alpha_1\cdot\alpha_2$. Then an isotopy of $\alpha_1\cdot\beta_1$ supported on the tubular neighborhood of $\iota_1$ would give us a loop $\gamma$ which is disjoint with $\iota_1\sqcup\iota_2$. And an isotopy of $\alpha_2\cdot\beta_2$ supported on the tubular neighborhood of $\iota_1$ would give us a new one-sided loop $\iota'$ which has fewer intersections with $\iota_1\sqcup\iota_2$. If $\gamma$ separate $\iota_1$ and $\iota_2$, the isotopy of $\alpha_1\cdot\bar{\beta_2}$ would give us a new loop which does not separate $\iota_1$ and $\iota_2$. We name this new loop by $\gamma$, then the isotopy of $\alpha_2\cdot\bar{\beta_1}$ would give us the new one-sided loop $\iota'$.

Now we have $\omega_f(\iota)=\omega_f(\iota')+\omega_f(\gamma)$. Since $\gamma$ is a loop on $\Sigma_{h,k+2}^+$, $\omega_f(\gamma)=2\omega_{f^2}(\gamma)$ is even in $Z_n$. If $h=0$, since $\gamma$ does not separate $\iota_1$ and $\iota_2$, $\gamma$ and some other boundary loops of $\Sigma_g/f-N(B_f)$ would bound a punctured sphere. Then $\omega_f(\gamma)$ is the sum of the value of $\omega_f$ at these boundary loops of $\Sigma_g/f-N(B_f)$. Hence we always have $2d|\omega_f(\gamma)$.

Repeat this process finite times, we would obtain a one-sided loop $\iota''$ which intersects with $\iota_1\sqcup\iota_2$ once. Suppose $\iota''$ intersects with $\iota_1$ once and the intersection is $\iota''(0)=\iota_1(0)$. Then an isotopy of $\iota''\cdot\iota_1$ supported on the neighborhood of $\iota_1$ would give us a loop $\gamma$ which is disjoint with $\iota_1\sqcup\iota_2$. If this loop separate $\iota_1$ and $\iota_2$, we let $\gamma$ be the isotopy of $\iota''\cdot\bar{\iota_1}$. Hence we have $\omega_f(\iota'')=\pm\omega_f(\iota_1)+\omega_f(\gamma)$ and $2d|\omega_f(\gamma)$.
\end{proof}

\begin{theorem}\label{s=2}
Given two orientation reversing periodic maps $f$ on $\Sigma_g$ and $f'$ on $\Sigma_g'$,
if $\Sigma_g/f$ and $\Sigma_g'/f'$ are both non-orientable and both have data with $s=2$, then $f$ and $f'$ are conjugate to each other if and only if the corresponding data are equivalent under (R0), (R2) and (R4).
\end{theorem}

Now we finish this section by some considerations about the relation between conjugacy and local conjugacy.
\begin{lemma}
The following numbers or facts are determined by local conjugacy:
\begin{enumerate}
  \item Number $n$, $b$ and $k$
  \item $\chi(\Sigma_g/f)$
  \item $\Sigma_g/f$ is orientable or not
  \item The genus of $\Sigma_g/f-N(B_f)$
  \item Number $d$ in relation (R4).
\end{enumerate}
\end{lemma}

\begin{proof}
The orbit length of $x\in N(W_f)-W_f$ equals the period $n$. $b$ and $k$ are determined by local conjugacy in an obvious way. Since $\chi(\Sigma_g-N(W_f))$ is $n$ times $\chi(\Sigma_g/f-N(B_f))$, then $\chi(\Sigma_g/f-N(B_f))+b=\chi(\Sigma_g/f)$.

$\Sigma_g/f$ is orientable iff. $W_f$ separates $\Sigma_g$, which is determined by local conjugacy.

The genus of $\Sigma_g/f$ equals $h$ when $s=0$ and equals $2h+s$ when $s\neq 0$. Hence it can be derived from $2-2h-k-s=\chi(\Sigma_g/f)$.

In particular if $f$ has data with $s=2$, then $h$ in this data is also determined. If $h=0$, $d=1$. Otherwise, since the value of $|\omega_f|$ at boundary loops of $\Sigma_g/f-N(B_f)$ is determined by local conjugacy, then $d=\gcd\{\theta_1,\cdots,\theta_b,\eta_1,\cdots,\eta_k\}$ is determined.
\end{proof}

\begin{corollary}\label{local}
If $f$ and $f'$ are locally conjugate, then $\Sigma_g/f$ is homeomorphic to $\Sigma_g'/f'$ and
\begin{enumerate}
 \item If $\Sigma_g/f$ is orientable, then $f$ must be conjugate to $f'$
 \item If $\Sigma_g/f$ is non-orientable with genus number odd, then
  \begin{enumerate}
   \item If $m$ is odd, $f$ must be conjugate to $f'$
   \item If $m$ is even, $f$ must be conjugate to $f'$ or $f'^{m+1}$
  \end{enumerate}
 \item If $\Sigma_g/f$ is non-orientable with genus number even, then $f$ may belongs to at most $\frac{\gcd\{d,m\}+1}{2}$ different conjugate classes.
\end{enumerate}
\end{corollary}

\begin{proof}
The first case is immediate from theorem \ref{s=0}.

If $s=1$ and $f$ has data $[h,2m;(2\theta_1,\cdots,2\theta_b);(2\eta_1,\cdots,2\eta_k);(\zeta)]$. Since $f'$ and $f$ are locally conjugate, by (R2), $f'$ has data $[h,2m;(2\theta_1,\cdots,2\theta_b);(2\eta_1,\cdots,2\eta_k);(\zeta')]$. By theorem \ref{reconstruct}, $\zeta'=\zeta \mod m$, hence $\zeta'=\zeta$ or $\zeta'=\zeta+m$. When $m$ is even, $f^{m+1}$ has data $[h,2m;(2\theta_1,\cdots,2\theta_b);(2\eta_1,\cdots,2\eta_k);(\zeta+m)]$; when $m$ is odd, $\zeta+m$ is even, by theorem \ref{reconstruct}, $[h,2m;(2\theta_1,\cdots,2\theta_b);(2\eta_1,\cdots,2\eta_k);(\zeta+m)]$ will not correspond to any periodic map. Then (2) is immediate.

If $s=2$ and $f$ has data $[h,2m;(2\theta_1,\cdots,2\theta_b);(2\eta_1,\cdots,2\eta_k);(\zeta_1,\zeta_2)]$. Since $f'$ and $f$ are locally conjugate, by (R2), $f'$ has data $[h,2m;(2\theta_1,\cdots,2\theta_b);(2\eta_1,\cdots,2\eta_k);(\zeta_1',\zeta_2')]$. By theorem \ref{reconstruct}, $\zeta_1+\zeta_2=\zeta_1'+\zeta_2' \mod m$. By theorem \ref{s=2}, $f$ and $f'$ are conjugate if $(\zeta_1',\zeta_2')=(\zeta_1+2d,\zeta_2-2d)$ or $(\zeta_1',\zeta_2')=(\zeta_1+2m,\zeta_2+2m)$ or $(\zeta_1',\zeta_2')=(\zeta_2,\zeta_1)$. Then it's easy to see that there are at most $\frac{\gcd\{d,m\}+1}{2}$ pairs $(\zeta_1',\zeta_2')$ which may give us different conjugate classes.
\end{proof}

\section{The orientation reversing periodic maps with large period}\label{large period}
As an application of our theory, we study the orientation reversing periodic maps with period $n>2(g-1)$. In particular we shall give a list of all possible conjugate classes of those with period $n\geq 3g$. Similar work was done in ~\cite{Zhao} by using techniques in NEC groups. In fact they have more detailed classification with $n>3g-3$.

\begin{lemma}\label{power2}
If the pre-data of $f$ is the data of $\varphi^r$, where $\varphi$ is an orientation preserving periodic map of period $m$ and $r$ is an integer satisfying $ur-vm=1$. Then
\begin{enumerate}
  \item If $u$ is odd, $f$ is the power of $f^u$, whose pre-data is the data of $\varphi$
  \item If $u$ is even, $f$ is the power of $f^{u+m}$, whose pre-data is the data of $\varphi$.
\end{enumerate}
\end{lemma}

\begin{proof}
By the definition of pre-data, $f^2_+$ on $\Sigma_{\tau,e}^+$ is conjugate to $\varphi^r$. Since $ur-vm=1$, $(f^2_+)^u$ is conjugate to $(\varphi^r)^u=\varphi$.

If $u$ is odd, $f^u$ is an orientation reversing periodic map whose pre-data is the data of $\varphi$. From $ur-vm=1$ and $u$ is odd, we know that $\gcd\{u,n\}=1$. Thus $f$ is a power of $f^u$.

If $u$ is even, from $ur-vm=1$ we know that $m$ is odd. Then $u+m$ is odd and $f^{u+m}$ is an orientation reversing periodic map whose pre-data is the data of $\varphi$. From $(u+m)r-(v+r)m=1$ and $u+m$ is odd, we know that $\gcd\{u+m,n\}=1$. Thus $f$ is a power of $f^{u+m}$.
\end{proof}

\begin{lemma}
Given data $[h,m;(\theta_1,\cdots,\theta_b);(\eta_1,\cdots,\eta_c)]$ of an orientation preserving periodic map on compact oriented surface. It is the pre-data of some orientation reversing periodic map on a closed surface if and only if:
\begin{enumerate}
 \item $c>0$
 \item $t_j=\gcd\{\eta_j,m\}$ is odd, $j=1,2,\cdots,c$
\end{enumerate}
\end{lemma}
\begin{proof}
Suppose that the above conditions are true.

If $m$ is even, we set $k=0$ and $s=c$. Since $t_j=\gcd\{\eta_j,m\}$ is odd, we know that $\eta_j$ is odd. Now we set $\zeta_l=\eta_l$, $l=1,\cdots,s$. By theorem \ref{reconstruct}, $[h,2m;(2\theta_1,\cdots,2\theta_b);(--);(\zeta_1,\cdots,\zeta_s)]$ is the data for some orientation reversing periodic maps.

If $m$ is odd, we set $k=c$ and $s=0$. Then by theorem \ref{reconstruct}, $[h,2m;(2\theta_1,\cdots,2\theta_b);(2\eta_1,\cdots,2\eta_k);(--)]$ is the data for some orientation reversing periodic maps.

The necessity follows from lemma \ref{t_i}.
\end{proof}

Given a collection of pre-data, it's easy to list all possible associated data of some orientation reversing maps. For simplicity, in the following discussion we only use pre-data.

Let $x_i=\frac{m}{\gcd\{\theta_i,m\}}$, $y_j=\frac{m}{\gcd\{\eta_j,m\}}$. By theorem \ref{Nielsen12} and theorem \ref{reconstruct}, we have the Riemann-Hurwitz formula

\begin{equation}\label{RH}
2\tau-2=m\left(2h-2+\Sigma_{i=1}^b(1-\frac{1}{x_i})+\Sigma_{j=1}^c(1-\frac{1}{y_j})\right)
\end{equation}
and
\begin{equation}\label{genus2}
g-1=m\left(2h-2+c+\Sigma_{i=1}^{b}(1-\frac{1}{x_i})\right).
\end{equation}

Then we have
\begin{equation}\label{genus}
g-1=2\tau-2+\Sigma_{j=1}^c\gcd\{\eta_j,m\}.
\end{equation}

\begin{lemma}\label{t=0}
If $n>2(g-1)$ and $y_j=1$, for some $j\in\{1,2,\cdots,c\}$, then $f$ must be conjugate to the power of a map of the following types:
\begin{tabbing}
1111\=$g=m-1$, $m$ is odd,11\= pre-data $[0,m;m-1,1;m]$\kill
1.  \>$g=m-1$, $m$ is odd,  \> pre-data $[0,m;(m-1,1);(m)]$\\
2.  \>$g=m-2$, $m$ is odd,  \> pre-data $[0,m;(m-1);(1,m)]$
\end{tabbing}
\end{lemma}

\begin{proof}
Since $y_j=1$ for some $j$, we know that $\gcd\{\eta_j,m\}=m$ for some $j$. Then by theorem \ref{reconstruct}, $m$ should be an odd integer.

From $n>2(g-1)$ and formula \eqref{genus} we have $$m>g-1=2\tau-2+\Sigma_{j=1}^c \gcd\{\eta_j,m\}\geq 2\tau-2+m+(c-1)$$ Hence $2>2\tau+(c-1)$, which implies that $\tau=0$, $c=1,2$.

If $\tau=0$, $c=1$, we have $g-1=-2+m$. It's easy to see that $f^2_+\colon\thinspace\Sigma_{0,m}\rightarrow\Sigma_{0,m}$ is a rotation on the punctured 2-sphere with two fixed point. Then $f^2_+$ is the is the power of a map with data $[0,m;(m-1,1);(m)]$. By lemma \ref{power2}, $f$ is the power of a map with pre-data $[0,m;(m-1,1);(m)]$.

If $\tau=0$, $c=2$, we have $m>g-1=-2+m+\gcd\{\eta_2,m\}$. Thus $\gcd\{\eta_2,m\}=1$. Similarly, $f$ is the power of a map with pre-data $[0,m;(m-1);(1,m)]$.
\end{proof}

\begin{lemma}\label{b+c}
If $n>2(g-1)$, we must have $h=0$ and the following possibilities:
\begin{tabbing}
0.11\=$c=2$,1\=$b=0,1$\kill
1.  \>$c=2$, \>$b=0,1$\\
2.  \>$c=1$, \>$b=1,2,3$
\end{tabbing}
\end{lemma}

\begin{proof}
First of all we would show that $b+c\geq 2$: Assume that $b+c<2$, by $c>0$ we have $b=0$, $c=1$. From theorem \ref{reconstruct} we have $m|\eta_1$ and $y_1=1$. By lemma \ref{t=0}, we know that $b+c=3$, which contradicts with the assumption.

Then by $n>2(g-1)$ and formula \eqref{genus2}, we have
$1>\frac{g-1}{m}=2h-2+c+\Sigma_{i=1}^{b}(1-\frac{1}{x_i})\geq 2h-2+c+\frac{b}{2}$.
Hence $h=0$ and (i)$c=2$, $b=0,1$ or (ii) $c=1$, $b=1,2,3$.
\end{proof}

Since $k+s=c$, we easily have:
\begin{theorem}
If $f$ is an orientation reversing periodic map of period $n$ on closed surface $\Sigma_g$ and $n>2(g-1)$, then $\Sigma_g/f$ is homeomorphic to $\Sigma_{0,2}$, $N_{1,1}$, $N_{2,0}$, $\Sigma_{0,1}$ or $N_{1,0}$.
\end{theorem}

\begin{lemma}\label{b+c=2}
If $h=0$, $b+c=2$ and $y_j>1$ for each $j=1,2,\cdots,c$, then $f$ is conjugate to the power of some map with the following pre-data:
\begin{tabbing}
0.11\= $g=0$,11\= pre-data $[0,m;(m-1);(1)]$\kill
1.  \> $g=0$,  \> pre-data $[0,m;(m-1);(1)]$\\
2.  \> $g=1$,  \> pre-data $[0,m;(--);(m-1,1)]$
\end{tabbing}
\end{lemma}

\begin{proof}
By the formula \eqref{RH} we have $2\tau-2=-\Sigma_{i=1}^{b}\gcd\{m,\theta_i\}-\Sigma_{j=1}^{c}\gcd\{m,\eta_j\}$.

Then we obtain $\tau=0$ and each term of the right side of this equation has value $1$, which implies that
\begin{tabbing}
0.11\= $g=0$,11111111\= pre-data $[0,m;m-1;1]$\kill
1.  \> $b=1$, $c=1$  \> $f^2_+\colon\thinspace\Sigma_{0,1}\rightarrow\Sigma_{0,1}$ is conjugate to a rotation on a disc\\
2.  \> $b=0$, $c=2$  \> $f^2_+\colon\thinspace\Sigma_{0,2}\rightarrow\Sigma_{0,2}$ is conjugate to a rotation on a annulus.
\end{tabbing}
The rest of the proof is trivial.
\end{proof}

\begin{lemma}\label{b+c=4}
If $n>2(g-1)$, $b+c=4$ and $y_j>1$ for each $j=1,2,\cdots,c$, then $f$ is conjugate to the power of some map with the following pre-data:
\begin{tabbing}
0.11\=$n=2g$,11\= pre-data $[0,2k;(k,k,2k-1);(1)]$\kill
1.  \>$n=2g$,  \> pre-data $[0,2k;(k,k,2k-1);(1)]$\\
2.  \>$n=2g$,  \> pre-data $[0,6;(3,4,4);(1)]$\\
3.  \>$n=2g$,  \> pre-data $[0,6;(3,2,4);(3)]$\\
4.  \>$n=2g$,  \> pre-data $[0,12;(6,8,9);(1)]$\\
5.  \>$n=2g$,  \> pre-data $[0,30;(15,20,24);(1)]$
\end{tabbing}
\end{lemma}

\begin{proof}
By lemma \ref{b+c}, we have $h=0$, $c=1$, $b=3$. By the formula \eqref{genus2} we have $1>\frac{g-1}{m}=2-\Sigma_{i=1}^3\frac{1}{x_i}$. Then we have $\Sigma_{i=1}^3\frac{1}{x_i}>1$. The solution of this inequality corresponding exactly to the triangle groups on 2-sphere. Suppose $x_1 \leq x_2 \leq x_3$, we have only the following possibilities:

(1)$x_1=2$, $x_2=2$, $x_3\geq 2$, $\frac{g-1}{m}=\frac{x_3-1}{x_3}$. By $g-1=2\tau-2+\gcd\{m,\eta_1\}$, we know that $g-1$ is odd. From $x_3-1|g-1$, we know that $x_3=2k$ for some integer $k\geq 1$. By corollary \ref{m}, $m=2k$. Since $\frac{m}{\eta_1}$ is odd and $\lcm\{x_1,x_2,y_1\}=m$, we have $x_1=2k$. Then $g=2k$, $\tau=k$. Notice that $\pi_{f^2}^{-1}(\sigma_1)$ is invariant under the action of $f$, we can find a proper power of $f^2_+$, suppose it to be $(f^2_+)^r$, such that $\omega_{(f^2)^r}(\sigma_1)=1$. From $\gcd\{\omega_{(f^2)^r}(S_{q_i}),m\}=\frac{m}{x_i}$ and $m|\Sigma_{i=1}^3\omega_{(f^2)^r}(S_{q_i})+\omega_{(f^2)^r}(\sigma_1)$ we can easily deduce that $(f^2)^r$ has pre-data $[0,2k;(k,k,2k-1);(1)]$.

(2)$x_1=2$, $x_2=3$, $x_2=3$, $\frac{g-1}{m}=\frac{5}{6}$. By corollary \ref{m}, we have $m=6$, $y_1=2,6$, $g=6$. By $g-1=2\tau-2+\frac{m}{y_1}$ we have $\tau=2$ when $y_1=2$; $\tau=3$ when $y_1=6$. The rest of the calculation is omitted.

(3)$x_1=2$, $x_2=3$, $x_3=4$, $\frac{g-1}{m}=\frac{11}{12}$. By corollary \ref{m}, we have $m=12$, $y_1=12$. Then $g=12$, $\tau=6$. The rest of the calculation is omitted.

(4)$x_1=2$, $x_2=3$, $x_3=5$, $\frac{g-1}{m}=\frac{29}{30}$. By corollary \ref{m}, we have $m=30$, $y_1=30$. Then $g=30$, $\tau=15$. The rest of the calculation is omitted.
\end{proof}
Now we introduce a result of Hirose~\cite{Hirose}. (In his article, the condition for the following lemma is $\tau\geq 3$ and $m\geq 3\tau$, but it's well-know that results for the situation $\tau<3$ are also contained in his conclusion as long as $h=0$ and the branch points number is $3$. In fact the proof in his article also applied for this general case.)
\begin{lemma}\label{Hirose}
If $m\geq 3\tau$, $h=0$ and the branch points number is $3$, then the genus $\tau$, the period $m$ and the multiplicity $Q=(q_1,q_2,q_3)$ of the branch points of periodic map $\varphi\colon\thinspace\Sigma_{\tau}\rightarrow\Sigma_{h}$ with $q_1\leq q_2\leq q_3$ should be one of the following. And for each case, $\varphi$ must be conjugate to the power of periodic map with total valency $V=(m,u_1/q_1+u_2/q_2+u_3/q_3)$, here $U=(u_1,u_2,u_3)$ are also listed:
 \begin{tabbing}
0.11\=$\tau$ is arbitrary,1111\= $m=4\tau+2$,11\= $Q=(2,2\tau+1,4\tau+2)$,11\= $U=(1,\tau,1)$\kill
1.  \>$\tau$ is arbitrary,    \> $m=4\tau+2$,  \> $Q=(2,2\tau+1,4\tau+2)$,  \> $U=(1,\tau,1)$\\
2.  \>$\tau$ is arbitrary,    \> $m=4\tau$,    \> $Q=(2,4\tau,4\tau)$,      \> $U=(1,2\tau-1,1)$\\
3.  \>$\tau=3k$,              \> $m=3\tau+3$,  \> $Q=(3,\tau+1,3\tau+3)$,   \> $U=(2,k,1)$\\
4.  \>$\tau=3k+1$,            \> $m=3\tau+3$,  \> $Q=(3,\tau+1,3\tau+3)$,   \> $U=(1,2k+1,1)$\\
5.  \>$\tau=3k\,or\,3k+1$,    \> $m=3\tau$,    \> $Q=(3,3\tau,3\tau)$,      \> $U=(1,2\tau-1,1)$\\
6.  \>$\tau=3k+2$,            \> $m=3\tau$,    \> $Q=(3,3\tau,3\tau)$,      \> $U=(2,\tau-1,1)$\\
7.  \>$\tau=4$,               \> $m=12$,     \> $Q=(4,6,12)$,         \> $U=(3,1,1)$\\
8.  \>$\tau=6$,               \> $m=20$,     \> $Q=(4,5,20)$,         \> $U=(3,1,1)$\\
9.  \>$\tau=9$,               \> $m=28$,     \> $Q=(4,7,28)$,         \> $U=(1,5,1)$\\
10. \>$\tau=12$,              \> $m=36$,     \> $Q=(4,9,36)$,         \> $U=(3,2,1)$\\
11. \>$\tau=10$,              \> $m=30$,     \> $Q=(5,6,30)$,         \> $U=(4,1,1)$
 \end{tabbing}
\end{lemma}

\begin{theorem}\label{b+c=3}
If $n\geq 3g$, $b+c=3$ and $y_j>1$, for each $j=1,2,\cdots,c$, then $f$ is conjugate to the power of some map with the following pre-data:
\begin{tabbing}
0.11\=$g=2k$,111111\= $n=4g+4$,11\= pre-data $[0,4k+2;(2k+1,2k);(1)]$\kill
1.  \>$g=2k$,      \> $n=4g+4$,  \> pre-data $[0,4k+2;(2k+1,2k);(1)]$\\
2.  \>$g=2k$,      \> $n=4g$,    \> pre-data $[0,4k;(2k,2k-1);(1)]$\\
3.  \>$g=2k+1$,    \> $n=4g-4$,  \> pre-data $[0,4k;(2k);(2k-1,1)]$, ($k>1$)\\
4.  \>$g=6k$,      \> $n=3g+6$,  \> pre-data $[0,9k+3;(6k+2,3k);(1)]$\\
5.  \>$g=6k$,      \> $n=3g$,    \> pre-data $[0,9k;(3k,6k-1);(1)]$\\
6.  \>$g=6k+2$,    \> $n=3g$,    \> pre-data $[0,9k+3;(6k+2,1);(3k)]$\\
7.  \>$g=6k+2$,    \> $n=3g+6$,  \> pre-data $[0,9k+6;(3k+2,6k+3);(1)]$\\
8.  \>$g=6k+2$,    \> $n=3g$,    \> pre-data $[0,9k+3;(3k+1,6k+1);(1)]$\\
9.  \>$g=6k+4$,    \> $n=3g$,    \> pre-data $[0,9k+6;(3k+2,1);(6k+3)]$\\
10. \>$g=6k+4$,    \> $n=3g$,    \> pre-data $[0,9k+6;(6k+4,3k+1);(1)]$\\
11. \>$g=1$,       \> $n=3g+1$,  \> pre-data $[0,4;(2);(1,1)]$\\
12  \>$g=4$,       \> $n=3g$,    \> pre-data $[0,6;(1,2);(3)]$\\
13. \>$g=8$,       \> $n=3g$,    \> pre-data $[0,12;(9,2);(1)]$\\
14. \>$g=12$,      \> $n=3g+4$,  \> pre-data $[0,20;(15,4);(1)]$\\
15. \>$g=18$,      \> $n=3g+2$,  \> pre-data $[0,28;(7,20);(1)]$\\
16. \>$g=24$,      \> $n=3g$,    \> pre-data $[0,36;(27,8);(1)]$\\
17. \>$g=20$,      \> $n=3g$,    \> pre-data $[0,30;(24,5);(1)]$\\
 \end{tabbing}
\end{theorem}

\begin{proof}
By formula \eqref{genus}, we have $g=2\tau-1+\Sigma_{i=1}^{c}\gcd\{m,\eta_i\}$.
From $\frac{2m}{3}\geq g$ we have
\begin{equation}\label{e}
\Sigma_{i=1}^{c}\gcd\{m,\eta_i\}\leq 2\left(\frac{m}{3}-\tau\right)+1
\end{equation}

Then we know that $m\geq 3\tau$. Following the lemma \ref{Hirose}, for each case we can compute $\gcd\{m,-\}=(m/q_1,m/q_2,m/q_3)$ and $\omega_{f^2}(-)=(u_1m/q_1,u_2m/q_2,u_3m/q_3)$.

Remembering that (i)$\gcd\{m,\eta_j\}$ is odd; (ii)$0<\gcd\{m,\theta_i\}<m$, then all the possibilities satisfying formula \eqref{e} are:
\begin{enumerate}
 \item $\tau$ is arbitrary, $m=4\tau+2$, $\gcd\{m,-\}=(2\tau+1,2,1)$, $\omega_{f^2}(-)=(2\tau+1,2\tau,1)$
  \begin{enumerate}
    \item $b=2$, $\gcd\{m,\theta_1\}=2\tau+1$, $\gcd\{m,\theta_2\}=2$, $\gcd\{m,\eta_1\}=1$; $g=2\tau$, $n=4g+4$, pre-data $[0,4\tau+2;(2\tau+1,2\tau);(1)]$
    \item $b=2$, $\gcd\{m,\theta_1\}=1$, $\gcd\{m,\theta_2\}=2$, $\gcd\{m,\eta_1\}=2\tau+1$; $g=4$, $n=3g$, pre-data $[0,6;(1,2);(3)]$
    \item $b=1$, $\gcd\{m,\theta_1\}=2$, $\gcd\{m,\eta_1\}=1$, $\gcd\{m,\eta_2\}=2\tau+1$; $g=1$, $n=3g+1$, pre-data $[0,4;(2);(1,1)]$
  \end{enumerate}
 \item $\tau$ is arbitrary, $m=4\tau$; $\gcd\{m,-\}=(2\tau,1,1)$, $\omega_{f^2}(-)=(2\tau,2\tau-1,1)$
  \begin{enumerate}
    \item $b=2$, $\gcd\{m,\theta_1\}=2\tau$, $\gcd\{m,\theta_2\}=1$, $\gcd\{m,\eta_1\}=1$; $g=2\tau$, $n=4g$, pre-data $[0,4\tau;(2\tau,2\tau-1);(1)]$ (or $[0,4\tau;(2\tau,1);(2\tau-1)]$ representing a orientation preserving periodic map $\varphi^{2\tau-1}$, where $\varphi$ has data $[0,4\tau;(2\tau,2\tau-1);(1)]$.)
    \item $b=1$, $\gcd\{m,\theta_1\}=2\tau$, $\gcd\{m,\eta_1\}=1$, $\gcd\{m,\eta_2\}=1$; $\tau> 1$, $g=2\tau+1$, $n=4g-4$, pre-data $[0,4\tau;(2\tau);(2\tau-1,1)]$
  \end{enumerate}
 \item $\tau=3k$, $m=3\tau+3$; $\gcd\{m,-\}=(\tau+1,3,1)$, $\omega_{f^2}(-)=(2\tau+2,\tau,1)$
  \begin{enumerate}
    \item $b=2$, $\gcd\{m,\theta_1\}=\tau+1$, $\gcd\{m,\theta_2\}=3$, $\gcd\{m,\eta_1\}=1$; $g=2\tau$, $n=3g+6$, pre-data $[0,9k+3;(6k+2,3k);(1)]$
    \item $b=2$, $\gcd\{m,\theta_1\}=\tau+1$, $\gcd\{m,\theta_2\}=1$, $\gcd\{m,\eta_1\}=3$; $g=2\tau+2$, $n=3g$, pre-data $[0,9k+3;(6k+2,1);(3k)]$
  \end{enumerate}
 \item $\tau=3k+1$, $m=3\tau+3$; $\gcd\{m,-\}=(\tau+1,3,1)$, $\omega_{f^2}(-)=(\tau+1,2\tau+1,1)$
  \begin{enumerate}
    \item $b=2$, $\gcd\{m,\theta_1\}=\tau+1$, $\gcd\{m,\theta_2\}=3$, $\gcd\{m,\eta_1\}=1$; $g=2\tau$, $n=3g+6$, pre-data $[0,9k+6;(3k+2,6k+3);(1)]$
    \item $b=2$, $\gcd\{m,\theta_1\}=\tau+1$, $\gcd\{m,\theta_2\}=1$, $\gcd\{m,\eta_1\}=3$; $g=2\tau+2$, $n=3g$, pre-data $[0,9k+6;(3k+2,1);(6k+3)]$
  \end{enumerate}
 \item $\tau=3k\,or\,3k+1$, $m=3\tau$; $\gcd\{m,-\}=(\tau,1,1)$, $\omega_{f^2}(-)=(\tau,2\tau-1,1)$

  $b=2$, $\gcd\{m,\theta_1\}=\tau$, $\gcd\{m,\theta_2\}=1$, $\gcd\{m,\eta_1\}=1$; $g=2\tau$, $n=3g$, pre-data $[0,3\tau;(\tau,2\tau-1);(1)]$ (or $[0,3\tau;(\tau,1);(2\tau-1)]$ when $u(2\tau-1)+v(3\tau)=1$ for some integers $u,v$. In this case it's representing an orientation preserving periodic map $\varphi^u$, where $\varphi$ has data  $[0,3\tau;(\tau,2\tau-1);(1)]$)
  
 \item $\tau=3k+2$, $m=3\tau$; $\gcd\{m,-\}=(\tau,1,1)$, $\omega_{f^2}(-)=(2\tau,\tau-1,1)$,

  $b=2$, $\gcd\{m,\theta_1\}=\tau$, $\gcd\{m,\theta_2\}=1$, $\gcd\{m,\eta_1\}=1$; $g=2\tau$, $n=3g$, pre-data $[0,9k+6;(6k+4,3k+1);(1)]$ (or $[0,9k+6;(6k+4,1);(3k+1)]$ when $k$ is even. In this case it's representing an orientation preserving periodic map $\varphi^{3k+1}$, where $\varphi$ has data $[0,9k+6;(6k+4,3k+1);(1)]$)
 \item $\tau=4$, $m=12$; $\gcd\{m,-\}=(3,2,1)$, $\omega_{f^2}(-)=(9,2,1)$

  $b=2$, $\gcd\{m,\theta_1\}=3$, $\gcd\{m,\theta_2\}=2$, $\gcd\{m,\eta_1\}=1$; $g=8$, $n=3g$, pre-data $[0,12;(9,2);(1)]$
 \item $\tau=6$, $m=20$; $\gcd\{m,-\}=(5,4,1)$, $\omega_{f^2}(-)=(15,4,1)$

  $b=2$, $\gcd\{m,\theta_1\}=5$, $\gcd\{m,\theta_2\}=4$, $\gcd\{m,\eta_1\}=1$; $g=12$, $n=3g+4$, pre-data $[0,20;(15,4);(1)]$
 \item $\tau=9$, $m=28$; $\gcd\{m,-\}=(7,4,1)$, $\omega_{f^2}(-)=(7,20,1)$

  $b=2$, $\gcd\{m,\theta_1\}=7$, $\gcd\{m,\theta_2\}=4$, $\gcd\{m,\eta_1\}=1$; $g=18$, $n=3g+2$, pre-data $[0,28;(7,20);(1)]$
 \item $\tau=12$, $m=36$; $\gcd\{m,-\}=(9,4,1)$, $\omega_{f^2}(-)=(27,8,1)$

  $b=2$, $\gcd\{m,\theta_1\}=9$, $\gcd\{m,\theta_2\}=4$, $\gcd\{m,\eta_1\}=1$; $g=24$, $n=3g$, pre-data $[0,36;(27,8);(1)]$
 \item $\tau=10$, $m=30$; $\gcd\{m,-\}=(6,5,1)$, $\omega_{f^2}(-)=(24,5,1)$

  $b=2$, $\gcd\{m,\theta_1\}=6$, $\gcd\{m,\theta_2\}=5$, $\gcd\{m,\eta_1\}=1$; $g=20$, $n=3g$, pre-data $[0,30;(24,5);(1)]$

\end{enumerate}
\end{proof}

Now under the condition $n>2(g-1)$, the only unclear cases are:
\begin{enumerate}
  \item $h=0$, $b=1$, $c=2$, $x_i>1$ for $i=1,\cdots,b$, $y_j>1$ for $j=1,\cdots,c$ and $n<3g$
  \item $h=0$, $b=2$, $c=1$, $x_i>1$ for $i=1,\cdots,b$, $y_j>1$ for $j=1,\cdots,c$ and $n<3g$.
\end{enumerate}

\begin{corollary}\label{4g-4}
If $g>3$ is odd, all the possible orientation reversing periodic maps with period larger than or equal to $3g$ must have period $4g-4$ and be conjugate to the power of some orientation reversing periodic map with data $[0,8k;(4k);(--);(2k-1,1)]$.
\end{corollary}
\begin{proof}
It's easy to see that the only possible case is in theorem \ref{b+c=3}: $g=2k+1$, $n=4g-4$, $b=1$, $c=2$, pre-data $[0,4k;(2k);(2k-1,1)]$.

The possible data corresponding to this pre-data are:\\
$D_1=[0,8k;(4k);(--);(2k-1,1)]$, $D_2=[0,8k;(4k);(--);(2k-1,4k+1)]$,\\
$D_3=[0,8k;(4k);(--);(6k-1,1)]$, $D_4=[0,8k;(4k);(--);(6k-1,4k+1)]$.

By the equivalent relations listed in section \ref{equivalent relation}, we have:\\
$D_1=[0,8k;(4k);(--);(2k-1,1)]\sim^{R(1)}[0,8k;(-4k);(--);(-2k+1,-1)]$\\
$\sim^{R(2)}[0,8k;(-4k);(--);(-2k-1,1)]=[0,8k;(4k);(--);(6k-1,1)]=D_3$ and\\
$D_2=[0,8k;(4k);(--);(2k-1,4k+1)]\sim^{R(1)}[0,8k;(-4k);(--);(-2k+1,-4k-1)]$\\
$\sim^{R(2)}[0,8k;(-4k);(--);(-10k-1,4k+1)]=[0,8k;(4k);(--);(6k-1,4k+1)]=D_4$.

It's easy to verify that if the map corresponding to $D_1$ is $f$, then $f^{4k+1}$ is exactly corresponding to $D_4$.
\end{proof}

\end{document}